\documentclass[11pt]{amsart}
\usepackage{amsmath, amsfonts, amssymb,amsthm}
\usepackage{euscript}
\usepackage[T1]{fontenc}
\usepackage{graphicx}
\usepackage{color}

\usepackage{hyperref}

\newtheorem{thm}{Theorem}[section]
\newtheorem{lem}[thm]{Lemma}

\newtheorem{prop}[thm]{Proposition}
\newtheorem{cor}[thm]{Corollary}

\theoremstyle{definition}
\newtheorem{defn}[thm]{Definition}

\newtheorem{ex}[thm]{Example}
\newtheorem*{rem}{Remark}

\DeclareMathOperator{\R}{\mathbb R}
\DeclareMathOperator{\C}{\mathbb C}

\DeclareMathOperator{\Z}{\mathcal Z}
\DeclareMathOperator{\I}{\mathcal I}

\DeclareMathOperator{\SR}{\mathcal K^0}
\DeclareMathOperator{\SRR}{\mathcal R^0}

\DeclareMathOperator{\SO}{\mathcal O}
\DeclareMathOperator{\Pol}{\mathcal P}
\DeclareMathOperator{\K}{\mathcal K}

\DeclareMathOperator{\PP}{\mathbb P}

\DeclareMathOperator{\supp}{supp}
\DeclareMathOperator{\Sing}{Sing}

\DeclareMathOperator{\Cent}{Cent}

\DeclareMathOperator{\RCent}{R-Cent}
\DeclareMathOperator{\RSp}{R-Spec}

\DeclareMathOperator{\dom}{dom}

\DeclareMathOperator{\Sp}{Spec}
\DeclareMathOperator{\ReSp}{R-Spec}
\DeclareMathOperator{\MSp}{\rm Max}
\DeclareMathOperator{\pol}{indet}

\DeclareMathOperator{\Ker}{Ker}

\DeclareMathOperator{\p}{\mathfrak{p}}
\DeclareMathOperator{\q}{{\mathfrak q}}
\DeclareMathOperator{\ir}{{\mathfrak r}}

\DeclareMathOperator{\Max}{\rm Max}
\DeclareMathOperator{\ReMax}{\rm R-Max}

\DeclareMathOperator{\m}{\mathfrak m}

\def\JRadCe {{\rm{Rad}^C}}
\def\JRad {{\rm{Rad}}}

\def\hyp {\textbf{\textrm{\textsc{(mp)}}}}

\begin{document}

\title[\tiny{Weak and semi normalization in real algebraic
  geometry}]{Weak and semi normalization in real algebraic
  geometry}
\author{Goulwen Fichou, Jean-Philippe Monnier and Ronan Quarez}
\thanks{The authors benefit from the support of the Centre Henri Lebesgue ANR-11-LABX-0020-01. The first author is deeply grateful to the UMI PIMS of the CNRS for its hospitality during part of this project.}

\address{Goulwen Fichou\\
Univ Rennes, CNRS, IRMAR - UMR 6625, F-35000 Rennes, France}
\email{goulwen.fichou@univ-rennes1.fr}

\address{Jean-Philippe Monnier\\
   LUNAM Universit\'e, LAREMA, Universit\'e d'Angers}
\email{jean-philippe.monnier@univ-angers.fr}

\address{Ronan Quarez\\
Univ Rennes\\
Campus de Beaulieu, 35042 Rennes Cedex, France}
\email{ronan.quarez@univ-rennes1.fr}
\date\today

\begin{abstract} We define the weak-normalization and the seminormalization of a real algebraic variety relative to its central locus. The study is related to the properties of the rings of continuous rational functions and hereditarily rational functions on real algebraic varieties. We provide in particular several characterizations (algebraic or geometric) of these varieties, and provide a full description of centrally seminormal curves. 
\end{abstract}

\maketitle


The present paper is devoted to the study of the seminormalization of a real algebraic variety. The complex setting is settled for more than fifty years, first in the analytic case in \cite{AN} for studying the Chow variety of a complex analytic variety, and more generally for scheme in \cite{AB}.
Recently, the concept of seminormalization
appears in the study of the singularities in the minimal model program by Koll\'ar and Kov\'acs \cite{KoKo}, and also in \cite{Ko}.
If the complex setting is well known, the different attempts in real geometry have been unsuccessful so far \cite{AB,MR}.
The two main difficulties one has to face in the real algebraic case are the complexity of the topology of the real points of algebraic varieties, notably with regards to centrality issues, and the intriguing behavior of rational functions (as illustrated in \cite{KN}), whose understanding is somehow the cornerstone of this work. 
These issues lead to two major features of our construction. 

Let us review the classical case. 
The weak-normalization of a complex analytic variety has been introduced by Andreotti \& Norguet \cite{AN} in order to study the space of analytic cycles associated with a complex algebraic variety.
The operation of weak-normalization consists in enriching the sheaf of holomorphic functions with those continuous functions which are also meromorphic. Later Andreotti \& Bombieri \cite{AB} defined the notion of weak-normalization in the context of schemes. For algebraic varieties, it consists roughly speaking of an intermediate algebraic variety between an algebraic variety $X$ and its normalization, in such a way that the weak-normalization of $X$ is in bijection with $X$. The construction goes by identifying in the normalization all points in the fibers over $X$. It gives rise to a variety satisfying a universal property among those varieties in birational bijection via a universal homeomorphism onto $X$.
The theory of seminormalization, closely related to that of weak-normalization, have been developed later by Traverso \cite{T} for
commutative rings, with subsequent work notably by Swan \cite{Sw} or
Leahy \& Vitulli \cite{LV} (see also \cite{Vcor}), 
with a particular focus on the algebraic approach and the study of the
singularities. 
Note that in the geometric context of complex algebraic
variety, weak-normalization and seminormalization lead to the same
notion. We refer to Vitulli \cite{V} for a survey on weak normality
and seminormality for commutative rings and algebraic varieties.

In the context of real geometry, the first occurrence of weak
normality or seminormality is the work by Acquistapace, Broglia and
Tognoli \cite{ABT} in the case of real analytic spaces. In \cite{MR}
the Traverso-style seminormalization of real algebraic varieties is studied by considering the ring of regular functions,
showing that such notion does not provide a natural universal
property. Seminormalization in the Nash context is introduced in \cite{Ra}.
Our aim in this paper is to provide appropriate definitions for weak
normalization and seminormalization in real algebraic
geometry, leading to natural universal properties. Two major
differences with the complex case appear. First, the normalization is
no longer surjective on the real closed points in general, but only on the central loci, namely the Euclidean closure of the set of non-singular points. Moreover, the
notions of weak-normalization and seminormalization that we consider
in the paper are distinct, the difference being witnessed by the
behavior of continuous rational functions on the central locus of real algebraic varieties.

\vskip 5mm
Dealing with continuous functions on real algebraic varieties, we are
interested mainly in the real closed points of real varieties. Note
that the real closed points of a quasi-projective variety defined over $\R$ are always included in an affine variety, so that we will restrict ourselves to this setting in the paper. In particular, consider now
a real algebraic variety as an algebraic subset of $\R^n$ (in the sense of \cite{BCR}).
The first focus on continuous rational function in real geometry is
due to Kreisel \cite{Kre} who proved that a positive answer to Hilbert
seventeenth problem of representing a positive polynomial as a sum of
squares of rational functions, can always be chosen among continuous
functions. Besides, Kucharz \cite{Ku} used this class of functions to
approximate as algebraically as possible continuous maps between
spheres, whereas Koll\'ar \& Nowak \cite{KN} initiated the proper
study of these functions, proving notably that the restriction of a
continuous function defined on a central real algebraic variety (in
the sense of \cite{BCR}), which
is also rational, does not remain rational in general. It is however
the case as soon as the ambient variety is nonsingular. As a
consequence, on a singular real algebraic variety $X$ one may consider the
ring $\SR(X)$ of continuous rational functions, or its subring
$\SRR(X)$ consisting of those continuous rational functions which
remain rational under restriction. This class, called hereditarily
rational in \cite{KN}, has been systematically studied in \cite{FHMM}
under the name of regulous functions. When the real algebraic variety
is no longer central, a rational function may admits several
continuous extensions to the whole variety. This is the reason why we
will consider continuous rational functions on the central locus of
algebraic varieties.
The continuous rational functions are now extensively studied in real geometry, we refer for example to \cite{KuKu1, KuKu2, FMQ, Mo} for further readings related to the subject of the paper.

\vskip 5mm

Let us sketch the construction of the weak and semi-normalizations.
The rational functions on a real algebraic variety $X$ that satisfy a monic polynomial equation with coefficients in the ring $\Pol(X)$ 
of polynomial functions on $X$, form the integral closure of
$\Pol(X)$ in the ring $\K(X)$ of rational functions
on $X$,
which is a finite module over $\Pol(X)$. This ring is the polynomial ring of the normalization $X'$ of $X$, coming with a finite birational morphism onto $X$. 
The resulting morphism from the complex points of $X'$ to the complex
points of $X$ is well-known to be closed with respect to the Zariski topology, hence surjective. However this is no more the case in restriction to the real points ! The {\it real} behavior of normalization, together with various other integral closures, has been investigated in \cite{FMQ-futur}. In particular, it has been noted there that the lack of surjectivity is related to the notion of central locus of a real algebraic variety.

So in the present paper, we push further the study and restrict ourselves to the central part of the real varieties. When requiring that the rational functions admit a continuous extension to the central locus $\Cent X$ of $X$, 
the integral
closure of $\Pol(X)$ in $\SR(\Cent X)$ is still a finite module over $\Pol(X)$, 
and therefore it coincides with the polynomial ring of a real algebraic variety. The study of the properties of this variety is the main subject of the paper.
We call this variety the weak-normalization $X^{w_c}$ of $X$ relative to the central locus of $X$. It
comes again with a finite birational polynomial morphism $\pi^{w_c}:X^{w_c}\to X$, which is an
homeomorphism for the Euclidean topology in restriction to the central
loci. We provide several characterizations of $X^{w_c}$, notably from
a geometric point of view that $X^{w_c}$ is the biggest intermediate
variety between $X$ and $X'$ whose central locus is in bijection with
$\Cent X$, or in an algebraic point of view introducing the notion of
centrally weakly subintegral extension of rings. It satisfies a universal property as follows.
\vskip 2mm
{\bf Theorem.}
{\it Let $X$ be an affine real algebraic variety.
For each affine real algebraic variety $Y$ together with a
finite birational map
$\pi:Y\rightarrow X$ then $\pi:\Cent Y\to \Cent X$ is bijective if and
only if the map $\pi^{w_c}:X^{w_c}\to X$ factorizes through $\pi$.}
\vskip 2mm
The justification for
calling $X^{w_c}$ the weak-normalization of $X$ comes from Theorem
\ref{propunivcentweak}, which illustrates that 
$X^{w_c}$ satisfies analogue properties in the real setting as 
the weak-normalization for complex algebraic varieties. Notice that
complex algebraic varieties are central i.e equal to their central
locus since the semi-algebraic dimension at any point is maximal (in
its irreducible component). Note also that this restriction to the
central locus is necessary is the real situation, because in general a
biggest real algebraic variety in finite birational bijection with a given real algebraic variety does not exist, as illustrated by Example \ref{ex-final} at the end of the paper.

Now, replacing the ring of continuous rational functions with the ring of hereditarily rational functions leads similarly to the definition 
of the seminormalization $X^{s_c}$ of $X$ relative to the central locus of $X$, whose ring of polynomial functions is given by 
the integral closure of $\Pol(X)$ in $\SRR(\Cent X)$. The study of the properties of this seminormalization, in comparison with the weak normalization, is the second main topic of the present paper.
The seminormalization of $X$ relative to its central locus is
an intermediate variety between $X$ and $X^{w_c}$, so that $X^{s_c}$ admits a
finite birational morphism $\pi^{s_c}$ onto $X$ which is an homeomorphism on the
central loci for the Euclidean topology.
It is moreover the biggest intermediate variety between $X$ and $X'$
whose polynomial functions are hereditarily rational on $\Cent X$. It
is also the biggest intermediate variety between $X$ and $X'$ whose
central locus is in bijection with $\Cent X$ with an inverse map which
is hereditarily rational. At
the level of algebra, we characterize $\Pol(X^{s_c})$ as the maximal 
centrally subintegral extension of $\Pol(X)$. Similarly for the universal property, we have~:
\vskip 2mm
{\bf Theorem.}
{\it Let $X$ be an affine real algebraic variety. 
For each affine real algebraic variety $Y$ together with a
finite birational map
$\pi:Y\rightarrow X$ then $\pi:\Cent Y\to \Cent X$ is bijective
with an inverse map which is hereditarily rational if and only if the map $\pi^{s_c}:X^{s_c}\to X$ factorizes through $\pi$.}

\vskip 2mm

The normalization, weak-normalization and seminormalization relative to the central locus of a real
algebraic variety are different in general. The latter two coincide on
varieties where all continuous rational functions are hereditarily rational, for
instance in the case of curves. We provide in this particular case a
full description of the singularities of centrally weakly normal curves, in
the spirit of \cite{Da} in the complex context.

Note also that, if it is clear from the very construction that
isomorphic real algebraic varieties have isomorphic weak normalization
and seminormalization, it is still true, similarly to the case of the
normalization, that biregular real algebraic varieties have biregular weak normalization and seminormalization (cf Theorem \ref{bireg}).

\vskip 5mm
At this stage of lecture, the reader should imagine that the weak
normalization and the seminormalization relative to the central loci
share very similar properties and somehow behave in a parallel
manner. We would like to focus however on a crucial difference. If
adding a continuous assumption in the construction of the
normalization naturally leads to a variety in (universal) bijection
(on the central loci), the additional assumption of hereditary
rationality implies moreover that the residual real local fields are isomorphic. This major difference is responsible for the fact that, as in the complex case, the seminormalization process
relative to the central locus commutes with localization at a prime
ideal, but this is no longer the case for the weak normalization process !


\vskip 5mm

The paper is organized as follows. First section is devoted to some preliminaries on real algebraic varieties, notably concerning central loci, normalization process or continuous rational functions. Second section deals with topological properties of integral morphisms in the real context, where the lying over and going-up properties are revisited. We introduce there the algebraic treatment of weak normalization and seminormalization in the Traverso style. Third section deals geometrically with finite birational maps, and details how (hereditary) continuous rational functions behaves under such maps. It is in the fourth section that we define the weak and seminormalization relative to the central locus, using (hereditary) continuous rational functions. We make the link with the algebraic notions developed in section three, and tackle the issue of the commutation with the localization. The final section gives a full treatment of the curves case. Note that we made the choice to provide examples illustrating the most intriguing phenomena, which we find of interest both for non real, but also real, geometers !

\vskip 10mm

{\bf Acknowledgment} :  The authors are deeply grateful to F. Acquistapace and F. Broglia for mentioning to them 
the potential study of weak-normalization for real algebraic variety via continuous rational functions, and to A. Parusi\'nski for useful discussions.

\section{Preliminaries on real algebraic varieties}
In this section we review the basic definition of a real algebraic variety together with the properties of its normalization, and recall the concept of continuous rational functions.

\subsection{Real algebraic sets and varieties}
\label{sectiongeoalg}  
For a commutative ring $A$ we denote by $\Sp A$ the Zariski spectrum of $A$,
the set of all prime ideals of $A$. We denote by $\MSp A$ the set of maximal ideals of $A$. 
In this work, we also consider the real Zariski spectrum $\ReSp A$ which consists
in all the real prime ideals of $A$. 

Recall that an ideal $I$ of $A$ is called
real if, for every sequence $a_1,\ldots,a_k$ of elements of $A$, then
$a_1^2+\cdots+a_k^2\in I$ implies $a_i\in I$ for $i=1,\ldots,k$. 

An affine real algebraic variety is a Zariski spectrum of the form
$\Sp (\R[x_1,\ldots,x_n]/I)$ where $I$ is an ideal of $\R[x_1,\ldots,x_n]$.
To a real algebraic variety given by the ideal 
$I$ in $\R[x_1,\ldots,x_n]$ one associates the 
real algebraic set $X=\Z(I)$ of all points in $\R^n$ 
which cancel any polynomial in $I$ i.e $X$ corresponds to
$\ReSp(\R[x_1,\ldots,x_n]/I)\cap \MSp(\R[x_1,\ldots,x_n]/I)$. 
Conversely, to any real algebraic set $X\subset \R^n$ one may associate
the real algebraic variety given by the ideal $\I(X)\subset \R
[x_1,\ldots,x_n]$ of all polynomials
which vanish at all points of $X$. Given an ideal $I$ of
$\R[x_1,\ldots,x_n]$ then the real Nullstellensatz
\cite[Thm. 4.1.4]{BCR} says that $\I(\Z(I))=I$ if and only if $I$ is a real
ideal. In particular, for a real algebraic set then $\I(X)$ is a real
ideal. 

We are interested in this text in the geometry of the real closed points of real algebraic varieties. In this context, it is natural to consider only varieties which are affine since almost all real algebraic
varieties are affine \cite[Rem. 3.2.12]{BCR}.
In the sequel we mostly consider algebraic sets (as in
\cite{BCR}) rather than real algebraic varieties (as in \cite{Man}, with an emphasis on $\R$-schemes) and unless specified, these algebraic sets are real.

\vskip 2mm

In complex affine algebraic geometry, polynomial and regular functions
coincide and thus we have a unique and natural definition of morphism
between complex algebraic sets. In the real setting no such natural
definition exists. Usually, real algebraic geometers prefer working
with the ring $\SO(X)$ of regular functions, i.e. rational functions with no real poles (see \cite[Sect. 3.2]{BCR} for details), rather than the ring of polynomial functions $\Pol(X)=\R
[x_1,\ldots,x_n]/I$ where $I=\I(X)$ on a real algebraic set $X$. In
this paper however, we work rather with the ring of polynomial
functions due to its better properties with respect to the
normalization process (see \cite{FMQ-futur}).

Let $X$ be a real algebraic set. We denote by 
$\m_x$ the maximal ideal of functions in $\Pol(X)$ that vanish at $x\in X$. With the real Nullstellensatz \cite[Thm. 4.1.4]{BCR}, we have a natural correspondence between the points of $X$ and the real maximal ideals of $\Pol(X)$.

Let $X\subset\R^n$ and $Y\subset\R^m$ be real algebraic sets. A
polynomial map from $X$ to $Y$ is a map whose
coordinate functions are polynomial. A polynomial map $\varphi:X\rightarrow Y$ induces an $\R$-algebra
homomorphism $\phi:\Pol(Y)\rightarrow \Pol(X)$ defined by
$\phi(f)=f\circ\varphi$. The map $\varphi\mapsto \phi$
gives a bijection
between the set of polynomial maps from $X$ to $Y$ and
the $\R$-algebra homomorphisms from $\Pol(Y)$ to $\Pol(X)$. We say that
a polynomial map $\varphi:X\rightarrow Y$ is an
isomorphism if $\varphi$ is bijective with
a polynomial inverse, or in another words if $\phi:\Pol(Y)\rightarrow
\Pol(X)$ is an isomorphism. We define the analog notions with regular
functions in place of polynomials ones. In that situation, an
isomorphism will be called a biregular isomorphism.
Unless specified, a map is polynomial.

\vskip 5mm

Let $X\subset\R^n$ be a real algebraic set. The complexification of
$X$, denoted by $X_{\C}$, is the complex algebraic set
$X_{\C}\subset\C^n$, whose ring of polynomial functions is
$\Pol(X_{\C})=\Pol(X)\otimes_{\R}\C$. As already mentioned, we have
$\Pol(X_{\C})=\SO(X_{\C})$. We say that $X$ is
geometrically smooth if 
$X_{\C}$ is smooth. Remark that if $X$ is irreducible, then
$X_{\C}$ is automatically irreducible because $X$ is an algebraic
set. The situation is different when we consider an affine real algebraic variety $X$, 
actually $X$
can be irreducible and $X\times_{\Sp\R}\Sp\C$ reducible when the set
of real points of $X$ is not Zariski dense in the set of complex
points, as illustrated by the example of the affine real algebraic
variety corresponding to the ideal $(x^2+y^2)$ in $\R[x,y]$. 

Let $\varphi:X\rightarrow Y$ be a polynomial map between real
algebraic sets. The tensor product by $\C$ of the morphism of
$\R$-algebras $\phi: \Pol(Y)\rightarrow\Pol(X)$ gives a morphism
of $\C$-algebras $\Pol(Y_{\C})\rightarrow \Pol(X_{\C})$ and by duality
we get a polynomial map $\varphi_{\C}:X_{\C}\rightarrow Y_{\C}$ called
the complexification of $\varphi$. Since $\I(X)$ and $\I(Y)$ are real
ideals then $\varphi$ is an isomorphism if and only if $\varphi_{\C}$
is an isomorphism. The situation is again different if we consider
affine real algebraic varieties. Indeed,  two non-isomorphic affine real algebraic varieties can be isomorphic over the
complex numbers, for example the empty conic $\Sp(\R[x,y]/(x^2+y^2+1))$ with the circle $\Sp(\R[x,y]/(x^2+y^2-1))$.

\subsection{Normalization and central locus}

Let $A\to B$ be an extension of rings. An
element $b\in B$ is integral over $A$ if $b$ is the root of a monic
polynomial with coefficients in $A$. By \cite[Prop. 5.1]{AM}, $b$ is
integral over $A$ if and only if $A[b]$ is a finite $A$-module. This
equivalence allows to prove that $A_B'=\{b\in B|\,b\, {\rm is\,
  integral\, over}\,A\}$ is a ring called the integral closure of $A$ in
$B$. The extension $A\to B$ is said to be integral if $A_B'=B$. In
case $A$ is reduced with a finite number of minimal prime ideals and
$B$ is the total ring of fractions $K$ of $A$ (see below), the ring $A_K'$ is
denoted by $A'$ and is simply called the integral closure of $A$.
The ring $A$ is called integrally closed (in $B$) if
$A=A'$ ($A=A_B'$).
If $A$ is the
ring of polynomial functions on an algebraic set $X$ over a
field $k$ then $A'$ is a
finite $A$-module \cite[Sec. 33, after Lemma 2]{Ma}.
In this situation then $A'$ is a finitely generated $k$-algebra and
so $A'$ is the ring of polynomial functions of an
algebraic set, denoted by $X'$, called the normalization
of $X$. We recall that a map $X\rightarrow Y$ between two
algebraic sets over a field $k$ is said finite if the ring morphism
$\Pol(Y)\rightarrow \Pol(X)$ makes $\Pol(X)$ a finitely generated
$\Pol(Y)$-module. The inclusion $\Pol(X)\subset \Pol(X')$ induces a finite map
which we denote by $\pi':X'\rightarrow X$, called the normalization map, which is a birational equivalence. We say that an
algebraic variety $X$ over a field $k$ 
is normal if its ring of polynomial functions is integrally
closed.

Let us give a property which we will need 
to deal with the normalization of a non irreducible algebraic variety.
Namely, one has to deal with coordinate rings which are not necessarily
domains but only reduced rings.

Recall first that if $A$ is a reduced ring with minimal prime ideals  
$\p_1,\ldots,\p_r$, then $(0)=\p_1\cap\ldots \cap \p_r$ and one has the canonical injections 
$A\rightarrow A_1\times\ldots\times A_r\rightarrow K_1\times \ldots\times K_r=K$
where $A_i=A/\p_i$ and $K_i$ is the fraction field of $A_i$ for any $i$. 
The product of fields $K$ is called the total ring of fractions of $A$ and 
the $A_i$'s are called the irreducible components of $A$. In case $A$
is the ring of polynomial functions $\Pol(X)$ on a real algebraic set
$X$ then we denote by $\K(X)$ the total ring of fractions. If
$X_1,\ldots,X_t$ denote the irreducible components of $X$ then $\K(X)$
is the product of fields $\K(X_1)\times\cdots\times\K(X_t)$. Notice
that $\K(X)$ corresponds also to classes of rational functions on $X$
i.e classes of functions without poles on a \textit{dense} Zariski open subset
of $X$. In the following, we prefer to call $\K(X)$ the ring of rational
functions on $X$.

\begin{prop} \cite[Prop. 1.2]{FMQ-futur} \label{ReducedIntegralClosure}
Let $A$ be a reduced ring with minimal prime ideals $\p_1,\ldots,\p_r$.
Then, $A'=A'_1\times\ldots\times A'_r$ where 
$A'$ is the integral closure of $A$ in its total ring of fractions $K$, and, for any $i$, $A'_i$ is
the integral closure of $A_i=A/\p_i$ in its ring of fractions $K_i$.
\end{prop}

\begin{cor} \label{ReducedIntegralClosureGeom}
Let $X$ be an algebraic set. If
$X_1,\ldots,X_t$ denote the irreducible components of $X$ then
$\Pol(X')=\Pol(X_1')\times\cdots\times\Pol(X_t')$ and
$X'$ is the disjoint union of $X_1',\ldots X_t'$.
\end{cor}

Notice that the previous corollary shows that it is sufficient to
understand the normalization for irreducible algebraic sets. It will
no longer be the case with the weak and semi-normalization studied in
the paper.




\vskip 5mm

For a real algebraic set $X\subset \R^n$, we say that $X$ is
geometrically normal if the associated complex algebraic set
$X_{\C}$ is normal. It is well known that $X$ is normal if and only if $X$ is geometrically normal. Note that, if the normality of $X$ implies that the ring of regular functions on $X$ is integrally closed, the converse is not true in general. Note also that biregular real varieties have biregular normalizations \cite[page 75]{BCR}. For more about the integral closure of the ring of regular functions on a real algebraic set, we refer to \cite{FMQ-futur}.

Recall that map $\varphi:Y\rightarrow X$ is birational if
$\varphi$ induces an isomorphism between $\K(Y)$ and $\K(X)$.
Note that the normalization of an algebraic set $X$ is the biggest
algebraic set finitely birational to $X$. More precisely, for any
finite birational map $\varphi:Y\rightarrow X$, there exists $\psi:X'\to Y$ such that $\pi'=\phi \circ \psi$.

\vskip 5mm

The normalization can be though as a kind of weak desingularization of an algebraic variety, but much closer to the original variety due to the finiteness property. Note however that stringy phenomena may appear in the real case.
For instance, the normalization of a cubic with an isolated point (e.g. given by the equation $y^2=t^2(t-1)$ in $\R^2$) is a smooth curve such that no real point lies over the singular point of the cubic.
This phenomenon leads us to consider the central loci of real algebraic sets.

\begin{defn}
\label{centrale}
Let $X$ be an algebraic set. If $X$ is irreducible, we define the central locus of $X$, denoted by $\Cent X$, to be the Euclidean closure $\overline{X_{reg}}^{eucl}$ in $X$ of the set $X_{reg}$ of non-singular points of $X$. Otherwise, if $X$ is an algebraic set with irreducible components
$X_1,\ldots,X_t$, we define the central locus of $X$ to be $\Cent X=\bigcup_{i=1}^t\Cent X_i$.

We say that $X$ is central if $X$ is equal to $\Cent X$.
\end{defn}

\begin{rem}
\begin{enumerate}
  \item By \cite[Prop. 7.6.2]{BCR}, the central locus of an
    irreducible algebraic set is
   the locus of points where the local semi-algebraic dimension is
   maximal.
  \item An algebraic set can be central without its irreducible
    components being so.
\end{enumerate}
\end{rem}

The normalization of any real algebraic curves is central. This is true since the normalization of a curve is even non-singular. However, it may happen that modifying a central curve via a finite birational map creates a non-central curve. Even worst, the normalization of a central surface may create isolated points! These pathologies, illustrated by the following examples, are the main reason why in the paper, we define a concept for weak and semi normalization of real algebraic varieties 
{\it relative} to the central locus.

\begin{ex}\label{grospoint2}
\begin{enumerate}
\item Let $C=\Z(y^4-x(x^2+y^2))$ in $\R^2$. The only singular
point of $C_{\C}$ is the origin, which is the intersection of a
real branch with two complex conjugated branches. In particular $C$ is central.

Consider the rational function $f=y^2/x$ on $C$, which satisfies the
integral equation $f^2-f-x=0$. Adding $f$ to the polynomial ring of
$C$ gives rise to an algebraic curve $Y$ with ring of polynomial functions
$$\Pol(Y)=\Pol(C)[y^2/x]\simeq\dfrac{\R[x,y,t]}{(y^4-x(x^2+y^2),t^2-t-x,xt-y^2,y^2t-(x^2+y^2))}$$
and since $y^2/x$ is integral over $\Pol(C)$ we get a finite birational
map $\pi:Y\to C$. 
Note that $Y$ may be embedded in $\R^2$ via the projection forgetting
the $x$ variable, giving rise to the cubic with an isolated point of
equation $y^2=t^2(t-1)$ in $\R^2$. The preimage by $\pi$ of the
singular point of $C$ consists of two real points and one of them is
the isolated point of $Y$.
Note that the polynomial function on $Y$ corresponding to the rational function $f$ is equal to $t$ and it takes different values at these two points.

\item We can elaborate on the previous example to construct a central
surface whose normalization is not central. Consider the surface $S=\Z((y^2+z^2)^2-x(x^2+y^2+z^2))$ in $\R^3$. Then $S$ is central with a unique singular point at the origin. Its complexification admits two complex conjugated curves crossing at the origin as singular set. The rational function $f=(y^2+z^2)/x$ satisfies the integral equation $f^2-f-x=0$. Let $Y$ be the surface in $\R^4$ admitting as ring of polynomial function $\Pol(Y)=\Pol(S)[(y^2+z^2)/x]$. We have
$$\Pol(Y)\simeq \dfrac{\R[x,y,z,t]}{((y^2+z^2)^2-x(x^2+y^2+z^2),t^2-t-x,xt-(y^2+z^2),(y^2+z^2)t-(x^2+y^2+z^2))}$$
and since
$(y^2+z^2)/x$ is integral over $\Pol(S)$ we get a finite birational map $\pi:Y\to S$. 
Note that $Y$ may be embedded in $\R^3$ via the projection forgetting the $x$ variable, giving rise to the surface defined by the equation $y^2+z^2=t^2(t-1)$ in $\R^3$. This surface is no longer central, with an isolated singular point at the origin. The preimage of the origin in $S$ consists of two points, the isolated point in $Y$ plus a smooth point in the two dimensional sheet of $Y$. Note that $Y$ is normal since its complexification is an (hyper)surface with a singular point. So $Y$ is the normalization of $S$.
Note that the polynomial function on $Y$ corresponding to the rational function $f$ is equal to $t$ and it has different values at these two points.
\end{enumerate}
\end{ex}

\subsection{Rational and continuous functions}

The intriguing behaviour of rational functions on a real algebraic set admitting a continuous extension to the whole algebraic set has been investigated in \cite{KN}. Among them, the special class of hereditarily rational functions is of special interest.

Let $X\subset \R^n$ be an
algebraic set.
A rational function $f\in \K(X)$ is regular on a Zariski-dense open
subset $U\subset X$ if there exist polynomial functions $p$ and $q$ on
$\R^n$ such that $\Z(q) \cap U=\emptyset$ and $f=p/q$ on $U$. The
couple $(U,f_{|U})$ is called a regular presentation of $f$.
The biggest such $U$ is called the domain $\dom (f)$ of $f$, and its complementary in $X$ is the indeterminacy locus $\pol(f)$ of $f$.

\begin{defn}
\label{defratcont}
Let $f:X\to \R$ be a continuous function. We say that
$f$ is a continuous rational function on $X$ if there exists a
Zariski-dense open subset $U\subset X$ such that $f_{|U}$ is
regular. We denote by $\SR(X)$ the ring of continuous rational functions on $X$.

Let $f:\Cent X\to \R$ be a continuous function. We say that
$f$ is a continuous rational function on $\Cent X$ if there exist a Zariski-dense open subset $U\subset X$ and a regular function $g$ on $U$ such that the restriction $f_{|U\cap \Cent X}$ is equal to the restriction $g_{|U\cap \Cent X}$.
The ring of continuous rational functions on $\Cent X$ is denoted by $\SR(\Cent X)$. 

A map $Y\to X$ between real algebraic sets $X\subset \R^n$ and $Y\subset \R^m$ is called rational continuous if its
components are rational continuous functions on $Y$.

A map $\Cent Y\to \Cent X$ between the central loci of real algebraic
sets $X\subset \R^n$ and $Y\subset \R^m$ is called rational continuous
if its components are rational continuous functions on $\Cent Y$.

\end{defn}

A typical continuous rational function is provided by the function defined by $(x,y)\mapsto
x^3/(x^2+y^2)$ on $\R^2$ minus the origin, and by zero at the
origin. The rational functions considered in Examples \ref{grospoint2}
admit also a continuous extension (by the value $1$) at their unique indeterminacy point.

Consider also the Cartan umbrella $X=\Z(z(x^2+y^2)-x^3)$ and the function
$f=x^3/(x^2+y^2)$ extended by $0$ along the $z$-axis. Then $f$ is rational
continuous on $X$ and so on $\Cent X$, but $f$ is also polynomial on $\Cent X$ since its coincides there with $z$. Of course $f$ and $z$ are different on $X$.
Remark that a polynomial function on a real algebraic set $X$ is the
restriction to $X$ of a unique polynomial function on $X_{\C}$ and
uniqueness comes from the Zariski density of $X$ in $X_{\C}$.

Note also that on a curve with isolated points like the cubic curve $\Z(y^2-x^2(x-1))$, a function regular on the one-dimensional branches can be extended continuously by any real value at the isolated points. In particular, the natural ring morphism
$\SR(X)\rightarrow \K(X)$ which sends $f\in \SR(X)$ to the
class $(U,f_{|U})$ in $\K(X)$, where $(U,f_{|U})$ is a regular
presentation of $f$, is not injective in general. 
However, restricting our attention to the central locus, note that the canonical map
$\SR(\Cent X)\rightarrow \K(X)$ is now injective.

Considering now a reducible algebraic set $X$ with irreducible components $X_1,\cdots,X_t$. By restriction, one gets a well defined natural morphism $\Xi_0:\SR(\Cent X)\to\SR(\Cent X_1)\times
\cdots\times\SR(\Cent X_t)$. Indeed, 
a function
$f\in\SR(\Cent X)$
is in particular rational on each irreducible component $X_i$ of $X$
as a rational function, and its restriction $f_{|\Cent X_i}$ is
continuous. Moreover, this morphism is clearly injective and one also has:

\begin{prop} \label{ratcontreducible}
Let $X$ be an algebraic set with irreducible components
$X_1,\cdots,X_t$. Then $\SR(\Cent X)$ is isomorphic to the subring
of $\SR(\Cent X_1)\times \cdots\times\SR(\Cent X_t)$ consisting of
$t$-tuples $(f_1,\ldots,f_t)$ such that $f_i$ and $f_j$ coincide on
the intersection $\Cent X_i\cap\Cent X_j$ for all
$i,j\in\{1,\ldots,t\}$.
\end{prop}

\begin{proof}
Denote by $A$ the subring of $\SR(\Cent X_1)\times \cdots\times\SR(\Cent X_t)$ consisting of
$t$-tuples $(f_1,\ldots,f_t)$ such that $f_i$ and $f_j$ coincide on
the intersection $\Cent X_i\cap\Cent X_j$ for all
$i,j\in\{1,\ldots,t\}$. 

The image of the natural morphism 
$\SR(\Cent X)\to\SR(\Cent X_1)\times
\cdots\times\SR(\Cent X_t)$ is included in $A$ since $\Cent X=\bigcup_{i=1}^t\Cent X_i$. 

Let us show the converse inclusion. An element in $A$ gives rise
to $t$ rational functions defined on $X_1,\ldots,X_t$ respectively, so
it coincides with a unique rational function on $X$. Moreover, this
rational function admits a continuous extension to $\Cent X$ by
definition of $A$ and since $\Cent X=\bigcup_{i=1}^t\Cent X_i$.
\end{proof}

Another stringy phenomenon appearing only on singular sets is illustrated by Koll\'ar example \cite{KN}. Consider the surface $S=\Z(y^3-(1+z^2)x^3)$ in $\R^3$. The continuous function defined by $(x,y,z)\mapsto ~^3 \sqrt{1+z^2}$ is regular on $S$ minus the $z$-axis, however its restriction to the $z$-axis is no longer rational. 
The continuous rational functions do not having this pathological behavior have been called hereditarily rational functions. Originally defined on the whole (central) algebraic set, one may even define hereditarily rational functions over semi-algebraic subsets as in \cite[Def. 1.4]{KuKu2}.
For our purposes, we will mainly consider hereditarily rational functions over central loci, in the sense of the following definition. Note that one has to be precocious in such a definition because the central locus is just a semialgebraic set in general. 

For a subset $A$ of an algebraic set, we denote by $\overline{A}^Z$ the closure of $A$ with respect to the Zariski topology.
We introduce a notion of centrality relative to a given algebraic set.

\begin{defn} Let $X$ be a real algebraic set and $V$ be an irreducible algebraic  subset of $X$. We say that $V$ is central in $X$ if the Zariski closure of $V\cap \Cent X$ is equal to $V$.
\end{defn}

The relative centrality of $V$ in $X$ assures than
$V$ intersects  sufficiently the central locus of $X$. This is the
case for example for the stick of the Whitney umbrella (i.e. the
$z$-axis in the surface of $\R^3$ defined by $x^2=y^2z$)
while it is not the case for the stick of the Cartan umbrella (i.e. the
$z$-axis in the surface of $\R^3$ defined by $x^3=(x^2+y^2)z$).
Beware that a plane cubic with an isolated point is central in the plane, even if it is not central by itself. 

\begin{defn}
  Let $X$ be an algebraic set. A rational function $f$ on $X$ is hereditarily rational on $\Cent X$ if 
  $f\in\SR(\Cent X)$ and, for every
irreducible algebraic subset $V$ central in $X$, the restriction $f_{|V\cap \Cent X}$ is {\it rational on $V$} in the following sense : there exist a Zariski-dense open subset $W$ of $V$ and a rational function $g$ on $V$ which is regular on $W$, such that the restrictions of $f_{|V\cap \Cent X}$ and $g$ to $\Cent X \cap W$ coincide.

We denote by $\SRR(\Cent X)$ the ring of hereditarily rational functions on $\Cent X$.

A map $\Cent Y\to \Cent X$ between the central loci of real algebraic
sets $X\subset \R^n$ and $Y\subset \R^m$ is called hereditarily rational if its
components are hereditarily rational functions on $\Cent Y$.
\end{defn}

The relative centrality well behaves with respect to irreducible components.

\begin{lem}\label{lem-red} Let $X$ be an algebraic set, and $V$ be central in $X$. There exists an irreducible component $Y$ of $X$ such that $V\subset Y$ and $V$ in central in $Y$.
\end{lem}

\begin{proof}
The relative centrality of $V$ in $X$ implies that the semialgebraic dimension of $V$ and of $V\cap \Cent X$ coincide. In particular there exists an irreducible component $Y$ of $X$ such that $\dim V=\dim (V\cap \Cent Y)$. Moreover $V\subset Y$ by irreducibility of $V$, and finally $V\cap \Cent Y$ is Zariski dense in $V$. 
\end{proof}

In the case of curves, the rings $\SR(\Cent X)$ and 
$\SRR(\Cent X)$ coincide. For non singular algebraic sets, it is known that any continuous rational function is also hereditarily
rational \cite{KN}, and this is even the case in the presence of isolated
singularities \cite{Mo}.

Remark that the canonical
map $\SRR(\Cent X)\rightarrow \K(X)$ is again injective.
Moreover, if $X$ has irreducible components $X_1,\cdots,X_t$, one gets, a natural injective morphism $$\SRR(\Cent X)\to\SR(\Cent X_1)\times
\cdots\times\SR(\Cent X_t).$$ One may describe its image.

\begin{prop} \label{regulureducible}
Let $X$ be an algebraic set with irreducible components
$X_1,\cdots,X_t$. Then, $\SRR(\Cent X)$ is canonically isomorphic to the subring
of $\SRR(\Cent X_1)\times \cdots\times\SRR(\Cent X_t)$ consisting of
$t$-tuples $(f_1,\ldots,f_t)$ such that $f_i$ and $f_j$ coincide on
the intersection $\Cent X_i\cap\Cent X_j$ for all
$i,j\in\{1,\ldots,t\}$.
\end{prop}

\begin{proof}
Denote by $A$ the subring of $\SRR(\Cent X_1)\times \cdots\times\SRR(\Cent X_t)$ consisting of
$t$-tuples $(f_1,\ldots,f_t)$ such that $f_i$ and $f_j$ coincide on
the intersection $\Cent X_i\cap\Cent X_j$ for all
$i,j\in\{1,\ldots,t\}$. By Proposition
\ref{ratcontreducible}, we have an injective morphism 
$$\SR(\Cent X)\to\SR(\Cent X_1)\times
\cdots\times\SR(\Cent X_t), ~~~f\mapsto (f_{|\Cent
  X_1},\ldots,f_{|\Cent X_t}).$$ 
  Let $f\in\SRR(\Cent X)$. Let $V_i$ be an
algebraic subset central in $X_i$ 
and denote the restriction $f_{|\Cent X_i}$ by $f_i$. Note that $V_i$ is central in $X$ since $\Cent X_i\subset\Cent X$ and thus
$f_{|V_i\cap\Cent X}$ is rational on $V_i$ because $f$ is hereditarily rational. It implies that $f_{|V_i\cap\Cent X_i}$ is also rational on $V_i$.
 Since $(f_i)_{|V_i\cap \Cent
  X_i}=f_{|V_i\cap \Cent X_i}$ then it follows that $(f_i)_{|V_i\cap \Cent
  X_i}$ is also rational on $V_i$. So $f_i\in\SRR(\Cent X_i)$ and we get indeed an
injective morphism $\SRR(\Cent X)\to A$.

To prove the surjectivity, we have to
check that, given $(f_1,\ldots,f_t)$ in $A$, the induced continuous
rational function $f$ on $\Cent X$ is still hereditarily rational. Let $V$ be an
irreducible algebraic subset central in $X$. There exists an $i$ such that $V\cap \Cent X_i$ is Zariski dense in $V$ by Lemma \ref{lem-red}, therefore $f_{|V\cap \Cent X_i}=(f_i)_{|V\cap \Cent
  X_i}$ is rational on $V$. As a consequence $f_{|V\cap \Cent X}$ is also rational
on $V$, so that $f$ is hereditarily rational on $\Cent X$.
\end{proof}


\section{Some topological properties of integral morphisms}
In real algebraic geometry, it is common to use various topologies,
like the Zariski topology or the Euclidean topology. When dealing with algebra, the same situation appears, and in this section we study topological properties of integral morphisms with respect to Zariski topology, the topology of the real spectrum, and the real Zariski topology.

The aim of this section is to deal with the general algebraic properties of integral ring homomorphisms 
between two reduced rings with same total ring of fractions. The results will be applied 
in the geometric settings in the following sections.

From now on, all our rings will contain ${\mathbb Q}$.

\subsection{Several topologies on a ring}
Our main interest is the study of the real Zariski topology which can be seen as a 
the real part of the classical Zariski topology. We also introduce the real spectrum topology since it has been 
intensively studied in the literature and hence it provides some tools to study the real Zariski topology. 

\subsubsection*{Zariski topology}
Let $A$ be a commutative ring and denote by $\Sp A$
the Zariski spectrum of $A$, i.e the set of all prime ideals of $A$. The set $\Sp A$ can be endowed with the Zariski topology 
whose basis of open subsets is given by the sets $D(a)=\{\p\in\Sp A\mid a\notin \p\}$ for $a\in A$. The 
closed subsets are given by 
the sets $V(I)=\{\p\in\Sp A\mid I\subset \p\}$ where $I$ is an ideal of $A$. 

Let us denote $\Max A\subset \Sp A$ the subset of all maximal ideals of $A$. 

\subsubsection*{Real spectrum topology}
To a commutative ring $A$ one may also associate a topological subspace $\Sp_r A$ 
which takes into account only prime ideals $\p$ whose residual field admits an ordering. Let us detail 
this construction a bit. 

An order $\alpha$ in $A$ is given by a real prime ideal $\p$ of $A$ (called
the support of $\alpha$ and denoted by $\supp(\alpha)$) and an ordering on the residue 
field $k(\p)$ at $\p$. An order can equivalently be given by a morphism $\phi$
from $A$ to a real closed field.

One has a natural support mapping $\Sp_r A\rightarrow \Sp A$ which sends $\alpha$ to $\supp(\alpha)$.

The value $a(\alpha)$ of $a\in A$ at the ordering $\alpha$ is just $\phi(a)$. 
The set of orders of $A$ is called the real spectrum of $A$ and denoted by $\Sp_r A$. 
It is empty if and only if $-1$ is a sum of squares in $A$.  
One endows $\Sp_r A$ with a natural topology whose open subsets are
generated by the sets $\{\alpha\in\Sp_rA\mid a(\alpha)>0\}$ where
$a\in A$. 
Let $\alpha,\alpha'$ be
two points of $\Sp_r A$, then we say that $\alpha$ is a specialization
of $\alpha'$ if $\alpha$ is in the closure of the singleton
$\{\alpha'\}$. We denotes this property by $\alpha'\rightarrow\alpha$.

For more details
on the real spectrum, the reader is referred to \cite{BCR}.\par

\subsubsection*{Real Zariski topology}
We also consider the set $\ReSp A$ which is just the image of the support mapping, namely  it consists
of all the real prime ideals of $A$. We endow it with the induced Zariski topology.

We set $D_R(a)=D(a)\cap (\ReSp A)$ and $V_R(I)=V(I)\cap (\ReSp A)$.

 Then, the closed subsets of $\ReSp A$ have the form $V_R(I)$ where $I$ is an ideal of $A$ 
 and a basis of open subsets is given by the subsets $D_R(a)$ for $a\in A$.
 
\subsubsection*{Functoriality}
Let $\phi:A\rightarrow B$ be a ring morphism. It canonically induces a map
$\psi:\Sp B\rightarrow \Sp A$ which is continuous for the Zariski topology.

It also induces a map
$\psi_r:\Sp_r B\rightarrow \Sp_r A$ which is continuous for the real spectrum topology.

And also, 
\begin{prop}
The morphism $\phi:A\rightarrow B$ induces a map
$\psi_R:\ReSp B\rightarrow \ReSp A$ which is continuous for the real Zariski topology.
\end{prop}

\begin{proof}
Let us see first that this is a well-defined map. Indeed, let $\q\in\ReSp B$ and $\p=\psi(\q)$. Then, there exists an ordering 
on $k(\q)$ that one may define by giving a morphism
$B/\q\rightarrow R$ into a real closed field $R$. Hence, one gets the following commutative diagram :

$$\begin{array}{ccccc}
A&\rightarrow&B\\
\downarrow&&\downarrow\\
A/\p&\rightarrow&B/\q&\rightarrow&R\\
\end{array}$$
which defines an ordering on $k(\p)$ and hence $\p$ is a real prime ideal.

The continuity comes from the following sequence of equalities :
$$\psi_R^{-1}(D_R(a))=\psi_R^{-1}(D(a)\cap \ReSp A)=\psi^{-1}(D(a)\cap
\ReSp A)\cap \ReSp B=$$ $$\psi^{-1}(D(a))\cap \psi^{-1}(\ReSp A)\cap \ReSp B=
D(\phi(a))\cap \ReSp B=D_R(\phi(a)).$$
 \end{proof}

 From now on, we will deal with ring extensions, namely $\phi$ will be injective.
 
\subsection{Lying over and going-up}
\begin{defn}
\label{rlo}
We say that a ring extension $\phi:A\rightarrow B$ satisfies the lying over property if $\psi$ is surjective.
Likewise, we say that $\phi$ satisfies the real lying over property
if $\psi_R$ is surjective. 
\end{defn}

Recall, for instance from \cite[Thm. 9.3]{Ma}, that an integral ring extension $\phi:A\rightarrow B$ satisfies the lying over property, and $\psi$ induces a map from $\Max B$ to $\Max A$ which is surjective.

One has also, induced by $\psi_R$, a map from $\ReMax B$ to $\ReMax A$ but the real counterpart 
of the last property is false in general, namely $\psi_R$ is not necessarily surjective.
For instance, the normalization map is surjective for complex algebraic sets
but this is no longer the case for real algebraic sets, as illustrated by the example of the cubic with an isolated singularity
$y^2-x^2(x-1)=0$. Indeed, its 
normalization has only complex points over the isolated point. 
The same example says also that the lying over property does not imply the real lying-over property.

\begin{defn}
We say that a ring extension $\phi:A\rightarrow B$ satisfies the going-up property if, for any 
couple of prime ideals $\p\subset \p'$ in $\Sp A$ and a prime ideal $\q\in \Sp B$ lying over $\p$, there exists
a prime ideal $\q'\in \Sp B$ over $\p'$ and such that $\q\subset \q'$. 
\end{defn}

The going-up property is stronger than the lying over property : it is obvious in the case where $A$ and $B$ are domains
and it follows from a theorem by Kaplansky in full generality. An integral ring extension $\phi:A\rightarrow B$ satisfies the going-up property (cf. \cite[Thm. 9.4]{Ma} for instance).

Note moreover that if a ring extension $\phi:A\rightarrow B$ is integral, then $\psi$ is a closed mapping with respect to the Zariski topology. The real counterpart of this fact is false, and this is one motivation to consider a real going-up property for the real spectrum.
\begin{defn}
We say that a ring extension $\phi:A\rightarrow B$ satisfies the going-up property for the real spectrum if, for any 
couple of points $\alpha,\alpha'\in\Sp_r A$ such that $\alpha'\rightarrow\alpha$ and a point 
$\beta'\in \Sp_r B$ lying over $\alpha'$, there exists
a point $\beta\in \Sp_r B$ over $\alpha$ and such that $\beta'\rightarrow\beta$. 
\end{defn}

We recall from \cite[Ch. 2, Prop. 4.2 and 4.3]{ABR}:

\begin{prop}\label{GoingupRealSpectrum}
Assume that the ring extension $\phi:A\rightarrow B$ is integral. Then, $\phi$ satisfies the real going-up property
for the real spectrum and $\psi_r$ is a closed mapping.
\end{prop}

Likewise, one may define a going-up property for real prime
ideals. Looking at the normalization of a non-central irreducible real algebraic
curve,  we see that integral extensions do not necessarily 
satisfy the real going-up property since they do not necessarily
satisfy the real lying over property. 
This leads us to restrict ourselves to central loci in the sequel.

\subsection{Central real lying over}

Let $X$ be an algebraic set. Then $\Pol(X)$ and $\SO(X)$ are both
reduced rings with a finite number of minimal prime ideals
that are all real ideals. To generalize to an abstract setting, we say
that a ring $A$ satisfies the condition \hyp\, if $A$ is reduced with a finite number of minimal primes
that are all real ideals.

Before giving the definition of the central locus of a ring, let us recall that, for a domain $A$ with fraction field $K$, one can see $\Sp_r K$ as the subset of orders of $\Sp_r A$ whose support is $(0)$. Likewise, when $A$ is satisfying \hyp\, with total fraction field $K$, 
one can see $\Sp_r K$ as the subset of orders of $\Sp_r A$ whose support is a minimal prime ideal of $A$.

\begin{defn}
  Let $A$ be a ring satisfying \hyp\, with total ring of fractions
  $K$.  We define the central locus of $A$,
denoted by $\Cent A$, to be the set of all points in $\Sp_r A$ which
belong to the closure of $\Sp_r K$. 
We say that $A$ is a central ring if $\Cent A=\Sp_r A$.
We denote by $\RCent A$ the subset of $\RSp A$ given by all supports
of points in $\Cent A$. 
\end{defn}

An illustrative example is given by the polynomial ring of the Whitney
umbrella $X=\Z(x^2-y^2z)$, for which the ideal $(x,y)$ defining the
stick of the umbrella belongs to $\RCent \Pol(X)$, even if the stick
is only half-contained in the central locus of $X$. Actually, consider
the semialgebraic half-curve $\gamma$ given by $t\mapsto (0,0,t)$ for
positive values of $t$, which is contained in the half-stick inside
the central locus of $X$. This half-curve $\gamma$ defines an element
$\alpha$ of the real spectrum of $\Pol(X)$ (whose support has zero set the line supporting $\gamma$ and ordering is the set of all one variable polynomials which are non negative for small $t\geq 0$). Since $\gamma$ is included in the central locus, we can draw a two-dimensional semialgebraic sheet $T$ on the umbrella containing $\gamma$ in its closure. 
This sheet may be associated to a point $\beta$
of the real spectrum of $K$ whose support is the nil ideal
and whose ordering is the set of all polynomials which are non negative for small $(x,y,t)\in T$.
As a consequence the point $\beta$ specializes to $\alpha$ and hence $\alpha$ lies in $\Cent \Pol(X)$.

\vskip 2mm


Notice for the sequel that if $\p$ is a minimal prime ideal of a ring
$A$ satisfying \hyp\ , then $\p$ belongs to $\RCent A$.

In the geometric setting, one has already defined the notion of
central locus of an
algebraic set (Definition \ref{centrale}).
Of course, when $X$ is an algebraic set, then $X$ is a central set if and only if $\Pol(X)$ is a central ring
(as it can be deduced from \cite[Prop. 7.6.2, 7.6.4]{BCR}).
For instance, if $X$ is a cubic with an isolated point, then $X$ is not a central set. Since there does not exist any arc included in the one-dimensional component with origin the isolated point, the ring $\Pol(X)$ is not central neither.
In fact, one gets a little bit more : for any algebraic set $X$, the maximal ideals of $\Pol(X)$ that lie
in the central locus can be associated to points in $\Cent X$. 

The real going-up for the real spectrum recalled in Proposition \ref{GoingupRealSpectrum} implies a real lying over with respect to the central loci. We first state a version for domains:

\begin{prop}\label{finitebirationalcentraldomains}
	Let $\phi:A\rightarrow B$ be an integral injective morphism of domains. Then,
	\begin{enumerate}
		\item The morphism
		$\psi_R$ induces a mapping from $\RCent B$ to $\RCent A$.
		\item If moreover $A$ and $B$ have same fraction field, then 
		$\psi_R$ induces a surjective mapping from $\RCent B$ onto $\RCent A$.
	\end{enumerate}
\end{prop}

\begin{proof}
	Property (1) comes from the continuity of $\psi_r$. Namely, let us start with $\q\in \RCent B$. There is
	an ordering $\beta\in\Sp_r B$ whose support is $\q$
	and moreover there is a specialization $\beta'\rightarrow \beta$ in $\Sp_r B$ such that 
	the support of $\beta'$ is the nil ideal.
	Taking inverse images, one gets 
	a specialization $\alpha'\rightarrow \alpha$ in 
	$\Sp_r A$. By injectivity of $\phi$, the support of $\alpha'$ is also the nil ideal and hence, the support of 
	$\alpha$ lies in $\RCent A$. It follows that
        $\psi_R(\q)=\supp(\alpha)\in\RCent A$.
	
	Let us prove (2). Let $\p'\in \RCent A$. So $\p'$ is the support of $\alpha'\in \Sp_r A$
	and moreover there exists $\alpha\in \Sp_r K\cap\Sp_r A$ 
	such that $\alpha\rightarrow \alpha'$.
	
	Since $A$ and $B$ have same fraction field $K$, there is $\beta\in \Sp_r B\cap\Sp_rK$ over $\alpha$.
	By the real going-up for the real spectrum (Proposition \ref{GoingupRealSpectrum}), one deduces the existence of 
	$\beta'\in \Sp_r B$ over $\alpha'$ such that $\beta\rightarrow \beta'$
	and hence $\beta'\in \Cent B$.
	It implies also that the support of $\beta'$ is a real prime ideal
	$\q'\in\RCent B$ lying over $\p'$. 
\end{proof}

Note that property (2) remains valid as soon as 
any ordering on the 
total ring of fractions
of $A$ does extend to an ordering on the total ring of fractions of $B$. 
This is the case, for instance, if the fraction fields extension has odd degree.
 
Now, we come to the reduced case. 
Recall first that (\cite[Th 9.3 (ii)]{Ma}), given 
an integral extension $A\rightarrow B$, 
a prime ideal of $B$ lying over a minimal prime ideal of $A$ is a minimal prime ideal of $B$.
Note also that a minimal prime ideal of $B$ does not necessarily lies over a minimal prime of $A$.

Naturally, we will moreover assume that our rings have real minimal primes, namely that they satisfy \hyp.
One main difference with 
Proposition \ref{finitebirationalcentraldomains} for domains
is that we do not necessarily get an induced morphism 
$\RCent B\rightarrow \RCent A$. Take the example
for $A$ of the coordinate ring of a cubic curve with isolated point and $B=A\times A/\p$ where $\p$ is the
maximal ideal corresponding to the isolated point.

Nevertheless, one gets such a morphism when
$A$ and $B$ have same total ring of fractions. This context will be sufficient for us in the paper.

\begin{prop}\label{finitebirationalcentral}
Let $\phi:A\rightarrow B$ be an integral injective morphism of rings
satisfying \hyp. If $A$ and $B$ have same total ring of fractions, then 
the morphism $\psi_R$ induces a mapping from $\RCent B$ to $\RCent A$ and this mapping is surjective.
\end{prop}

\begin{proof}
It suffices to argue on each component and use 
Proposition \ref{finitebirationalcentraldomains}.

Namely, let $\p_1,\ldots,\p_r$ all the minimal primes of $A$
and 
$\q_1,\ldots,\q_s$ all the minimal primes of $B$.
Since $A$ and $B$ have same total field of fractions, one gets $r=s$ and, up to re-indexation, $A/\p_i\rightarrow B/\p_i$ is an integral extension of domains
having the same fraction field.

As a particular consequence, any real minimal prime in $B$ lies over a real minimal prime in $A$.
\end{proof}

Our standard geometric setting will be when $A=\Pol(X)$ and 
$B=\Pol(Y)$ are the polynomial rings of two given algebraic subsets $X$ and $Y$ together with a birational
polynomial mapping $Y\rightarrow X$.

The fact that the maximal ideals of $\Pol(X)$ that lie
in the central locus can be associated to points in $\Cent X$ can be recovered by the 
following key lemma which relates the geometric and algebraic notions of centrality. Let us emphasize that its proof is really of a semi-algebraic nature, generalizing the example of the Whitney umbrella exposed above.
\begin{lem}\label{UltrafiltersForRCenter}
	Let $V\subset X\subset \R^n$ be two algebraic sets where $V=\Z(\p)$ with $\p$ a non-zero real prime
	ideal of $\Pol(X)$.
	Then, $\p\in \RCent \Pol(X)$ if and only if $\Z(\p)$ is central in $X$.
\end{lem}
\begin{proof}
	Let us assume first that $V=\Z(\p)$ is central in $X$. By Lemma \ref{lem-red}, we may assume $X$ is irreducible.
	Note that then $V$ is also the Zariski
	closure of $\Cent V\cap \Cent X$.
	Set $T=\Cent X$ and $S=\Cent V\cap\Cent X$. Note that $S$ and $T$ are two closed semi-algebraic subsets of $X$.
	Our aim is to exhibit two orderings $\alpha$ and $\beta$ 
	respectively represented as ultrafilters in $S$ and $T$ and such that $\alpha$ is a specialization of $\beta$.
	To do so, we refer to the description of orderings in
	$\R[x_1,\ldots,x_n]$ as ultrafilters of semi-algebraic sets given in
	\cite[Prop. 7.2.4 and Rem. 7.5.5]{BCR}.
	
	Let $x$ be an arbitrary point in $S$. The question being local and semi-algebraic, 
	up to a semi-algebraic triangulation (\cite[Thm. 9.2.1]{BCR}), 
	one may assume that there is a semi-algebraic neighborhood $U$ of $x$ which can be taken to be the origin 
	of $\R^n$,
	$S$ contains $U\cap ((\R^+)^{\dim V}\times 0)$
	and $T$ contains $U\cap (\R^+)^{\dim X}$. It is then classical to construct an ordering
	$\alpha$ whose support is $\p$ and an ordering $\beta$ whose support is $(0)$ such that $\alpha$ specializes $\beta$ 
	(which itself specializes to $x$). It shows that $\p\in \RCent \Pol(X)$.
	
	Let us assume now that $\p\in \RCent \Pol(X)$. 
	Take $\alpha$ an ordering whose support is $\p$ and $\beta$ another
	ordering whose support is a minimal prime ideal $\q$ of $\Pol(X)$ and such that 
	$\beta$ specializes into $\alpha$. Replacing $X$ by the irreducible
	component of $X$ corresponding to $\q$ then we may assume $X$ is
	irreducible and $\q=(0)$ and it is clearly sufficient to consider that
	situation.
	Note first that $\alpha$ specializes to a maximal point $\gamma$ of the real spectrum $\Pol(X)$ but 
	this $\gamma$ does not necessarily correspond to a geometric point, for instance $\gamma$ could 
	corresponds to a branch going to infinity.
	Nevertheless, one may use the same arguments as previously. Thinking at points of the real spectrum of $\Pol(X)$ as
	ultrafilters of semi-algebraic subsets in $X$ (see 
	again \cite[Prop. 7.2.4 and Rem. 7.5.5]{BCR})), there are semi-algebraic subsets $A\subset V$ and $B\subset \Cent X$ representing $\alpha$ and $\beta$ such that $\overline{A}^Z=V$ and $A\subset B$. As a consequence 
	$$V=\overline{A}^Z\subset \overline{V\cap \Cent X}^Z\subset V$$
therefore $V$ is central in $X$	as required.
\end{proof}

\begin{defn} Let $\pi : Y \to X$ be a finite birational map and $V\subset X$ be an irreducible algebraic subset. We say that an irreducible algebraic subset $W$ of $Y$ lies over $V$ if $\mathcal I(W)$ lies over $\mathcal I(V)$ in the integral extension $\Pol(X) \to \Pol(Y)$.
\end{defn}

The following Lemma states that for a finite birational map with target $X$ then,
generically, the fibers over points in $V\cap\Cent X$, for $V$ central in $X$, are contained in the union of the central sets lying over $V$. This result will be crucial for the proof of two important results, Proposition \ref{prop-her} and Theorem \ref{thm-loca}.

\begin{lem}\label{lemm1}
	Let $\pi : Y \to X$ be a finite birational map and $V\subset X$ be an irreducible algebraic set central in $X$. Let $W_1,\ldots,W_r$ be the irreducible algebraic subsets of $Y$ lying over $V$ and central in $Y$. There exists an algebraic subset $Z\subset V$ with $\dim Z<\dim V$ such that the following inclusion holds :
	$$\pi^{-1}\big( (V\setminus Z)\cap \Cent X \big) \cap \Cent Y ~~\subset ~~W_1\cup \cdots \cup W_r.$$
\end{lem}

\begin{proof}
	Consider the decomposition of the algebraic subset $\pi^{-1}(V)$ of $Y$ into irreducible components. 
	Let $W$ be one of this components. Then $\dim W \leq \dim V$ by finiteness of $\pi$, and assume that the dimension of the semialgebraic set $W \cap \Cent Y$ is $\dim V$. Then this dimension is equal to $\dim W$ too, so that $W$ is central in $Y$ by Lemma \ref{UltrafiltersForRCenter}. In particular $W$ is equal to one of the $W_i$, with $i\in \{1,\ldots,r\}$.

	Denote by $W'$ the union of the irreducible components of $\pi^{-1}(V)$ different from $W_1,\ldots,W_r$. Then we have just proven that 
	$$\dim W' \cap \Cent Y < \dim V.$$ 
	Thus the algebraic subset $Z$ of $V$ defined by 
	$$Z=\overline{\pi (W' \cap \Cent Y )}^Z$$
	satisfies the required conditions.
\end{proof}

\begin{ex} Consider the normalization $X'=\mathcal Z(z^2-(1+x^2))\subset \R^3$ of the surface defined by $X=\mathcal Z(z^2-(1+x^2)(x^2+y^2)^2)$ in $\R^3$. Note that $X$ together with $X'$ are central, $X$ has a unique singular point at the origin, and that $\pi :X'\to X$ is defined by $(x,y,z)\mapsto (x,y,z(x^2+y^2))$.

	There is a unique irreducible curve $W$ lying over the irreducible curve $V\subset X$ defined by $x=0$ and $z=y^2$, and $W$ is the line given by $x=0$ and $z=1$. However the preimage of the origin in $V$ consists of the two points $(0,0,\pm 1)$ and that $(0,0,-1)$ does not belong to $W$. Here $Z$ consists of the origin in $V$.
\end{ex}

To end this section, let us note that from Proposition
\ref{finitebirationalcentral} and Lemma \ref{UltrafiltersForRCenter},
we deduce some central lying over properties for integral extensions
of geometric rings. 
\begin{prop}
\label{normalizationcentral2}
Let $\pi:Y\rightarrow X$ be a finite birational map between
algebraic sets.
Then the mapping $\RCent \Pol(Y)\to\RCent \Pol(X)$ is surjective.
When $X$ is central, $\Pol(X)\rightarrow \Pol(Y)$ satisfies the real lying over property.
\end{prop} 

In the forthcoming subsection, we come back to the algebraic setting and study a property
stronger than the central lying over.

\subsection{Central subintegral extensions}\label{sect-CSE}
Recall (\cite{V}) that an integral extension $A\rightarrow B$ 
is said subintegral if for any prime ideal $\p\in\Sp A$, there exists a unique
prime ideal $\q\in\Sp B$ lying over $\p$ and, moreover, the induced injective map on
the residue fields $k(\p)\rightarrow k(\q)$ is an isomorphism. 

One natural way to define a real subintegral extension $A\rightarrow B$ would be : 
it is an integral extension and given any real prime ideal $\p$ of $A$, there exists a unique
real prime ideal $\q$ of $B$ lying over $\p$ and the induced map on
the residue fields $k(\p)\rightarrow k(\q)$ is an isomorphism.
Due to centrality issues, we have to take into account central loci also and this leads to the following:

\begin{defn}
\label{defsrgu}
Let $A\rightarrow B$ be an integral extension of rings satisfying \hyp.
\begin{enumerate}
\item We say that $A\rightarrow B$ is centrally weakly subintegral ($w_c$-subintegral for short)
if, given any real maximal ideal $\p\in\RCent A\cap \Max A$, there exists a unique
real maximal $\q\in\RCent B\cap\Max B$ lying over $\p$ and the induced injective map on
the residue fields $k(\p)\rightarrow k(\q)$ is an isomorphism. 
\item We say that $A\rightarrow B$ is centrally subintegral ($s_c$-subintegral for short) if,
given any real prime ideal $\p\in\RCent A$, there exists a unique
real prime $\q\in\RCent B$ lying over $\p$ and the induced injective map on
the residue fields $k(\p)\rightarrow k(\q)$ is an isomorphism.
\end{enumerate}
\end{defn} 

The definitions of $w_c$ and $s_c$-subintegral extensions will lead to the notion of weak-normalization and seminormalization relative to the central loci introduced in \ref{WeakSemicentralLocus}.
It is worth mentioning that replacing prime ideals with maximal
ideals in the definition of subintegral extension gives rise to the same notion, at least for geometric rings.
On the contrary, if a $s_c$-subintegral extension is indeed a $w_c$-subintegral extension, the converse is false in general. Example \ref{K-surf} illustrates for instance that we may keep a bijection at the level of prime ideals, but loosing the equiresiduality condition. Example \ref{Ex-primepasbij} shows that we may loose the bijection between central real prime ideals too.

Centrally (weakly) subintegral extensions are stable under composition:
\begin{prop}\label{subintegralcomposition}
 Let $A\stackrel{\phi}{\rightarrow} B\stackrel{\psi}{\rightarrow} C$
 be two integral extensions of rings.
 Assume either that $A,B, C$ are domains, or $A,B,C$ satisfy \hyp and have same total ring of fractions. One has:
 \begin{enumerate}
  \item If $\phi$ and $\psi$ are both $w_c$-subintegral, then $\psi\circ\phi$ is also $w_c$-subintegral, 
  \item If $\phi$ and $\psi$ are both $s_c$-subintegral, then $\psi\circ\phi$ is also $s_c$-subintegral. 
 \end{enumerate}
\end{prop}
\begin{proof}
 The existence and the equiresiduality properties are clear by transitivity.
 The uniqueness property comes from Propositions \ref{finitebirationalcentraldomains}.1) and \ref{finitebirationalcentral}.
\end{proof}

In \cite{FMQ-futur} has been introduced the concept of
biregular extensions of rings ; it happens that these extensions are examples of
$w_c$ and $s_c$-subintegral extensions. Let us explain this fact now.

Recall that the 
ring $\SO(A)$ 
of
regular fractions of elements in $A$ is obtained from $A$ by inverting all elements in $1+\sum A^2$, and moreover
that $\SO(\Pol(X))$ coincide with the usual ring of regular functions when $X$ is an algebraic set.


In case $\SO(A)\to\SO(B)$ is an isomorphism we say that the extension $A\to B$ is biregular.

When $A\to B$ is an integral extension of rings satisfying \hyp, by \cite[Prop. 4.11]{FMQ-futur}, we know that
the extension $A\to B$ is biregular if and only if one
of the following equivalent conditions holds:
\begin{enumerate}
	\item Given any  ideal $\m\in\ReMax A$, there exists a unique
	maximal $\m'\in\Max B$ lying over $\m$ and moreover $\m'$ is real and the
	map $A_{\m}\to B_{\m'}$ is an isomorphism. 
		\item 
		Given any real prime ideal $\p\in\RSp A$, there exists a unique
		prime $\q\in\Sp B$ lying over $\p$ and moreover $\q$ is real and the
		map $A_{\p}\to B_{\q}$ is an isomorphism. 
		
\end{enumerate}

It is clear from this characterization that: 
\begin{cor} \label{implbireg}
 Let $A\to B$ be a biregular integral extension of rings satisfying
 \hyp. Then $A\to B$ is $s_c$-subintegral and $w_c$-subintegral.
\end{cor}


\subsection{Weak-normalization and seminormalization relative to the central locus of a ring}\label{CentrallyNormalizationRing}
In close relation with the notion of normalization, Traverso
\cite{T} has introduced the seminormalization of a ring $A$ with
integral closure denoted by $A'$, to be the ring $$^+A=\{f\in A'|\, \forall\p\in\Sp
A,\,f_{\p}\in A_{\p}+\JRad(A'_{\p})\}$$ 
where $f_{\p}$ is the image of $f$ by $A'\to A'_{\p}$ and where
$\JRad(B)$ denotes the Jacobson radical of $B$, i.e. the intersection
of all maximal ideals in the ring $B$.
Notice that for a ring $C$ and a multiplicative closed subset $S$ of
$C$, there is no confusion in writing $C'_{S}$ since $(C_S)'\simeq
(C')_S$ (\cite[Prop. 5.12]{AM}).

Inspiring from \cite{AN}, where only complex analytic varieties are considered, one may also define 
the weak-normalization of $A$ as the ring $$\{f\in A'|\, \forall\m\in\Max
A,\,f_{\m}\in A_{\m}+\JRad(A'_{\m})\}.$$

Note that for finite type algebras over an algebraically closed field of characteristic zero, these two notions of semi and weak normalization coincide as it is shown in \cite[Thm. 2.2]{LV} (when the characteristic is not zero see \cite{V}).

We are going to define real counterparts of these two notions, relatively to the central locus of a ring ; it is worth to advertise already that in our context, the two notions will be distinct.

We start with defining, for any ring $A$, the central radical $\JRadCe A$ of $A$ to be the intersection 
of all real maximal ideals $\m$ which are also in $\RCent A$. 
In case the ring $A$ is central, the central radical of $A$ coincides with the real Jacobson radical of $A$, i.e. 
the intersection of all real maximal ideals.

Now, we particularly focus on properties of integral
extensions of rings
contained in the integral closure. 
Let $A$ be a ring satisfying \hyp\, and with minimal prime ideals
$\p_1,\ldots,\p_t$. By Proposition \ref{ReducedIntegralClosure}, the
total ring of fractions $K$ of $A$ is given by
$K=\prod_{i=1}^tk(\p_i)$ (where $k(\p_i)$ denotes the residue field at
$\p_i$) and the integral closure $A'$ of $A$ is given by
$\prod_{i=1}^t (A/\p_i)'$. It follows that $A'$ satisfies also \hyp,
the minimal prime ideals of $A'$ are in bijection with those of $A$
and $K$ is also the total ring of fractions of
$A'$. The same properties remain for any intermediate ring between $A$
and $A'$ which yield:
\begin{prop}\label{finitebirationalcentralintclos}
	 Let $A\rightarrow B\rightarrow A'$ be a sequence of ring extensions
	where $A$ satisfies \hyp\, and $A'$ is the integral closure  of $A$.
	Then, one has a surjective mapping from $\RCent B$
	to $\RCent A$.
\end{prop}
\begin{proof}
	According to \cite[Lem. 2.8]{FMQ-futur}, $B$ satisfies also \hyp, the
	minimal prime ideals of $B$ and $A$ are in bijection 
	and $A\to B$ induces an isomorphism between the total rings
	of fractions.
	It suffices then to apply Proposition
	\ref{finitebirationalcentral}. 
      \end{proof}

We prove that sub-extensions of (weakly) centrally subintegral extensions contained
in the integral closure are (weakly) centrally subintegral.
\begin{prop}
  \label{intermedwc}
  Let $A\to C\to B\to A'$ be a sequence of ring extensions where $A$
  satisfies \hyp\, and $A'$ is the integral closure  of $A$. One has:
  \begin{enumerate}
    \item If $A\to
  B$ is $w_c$-subintegral then $A\to
  C$ is also $w_c$-subintegral.
\item If $A\to
  B$ is $s_c$-subintegral then $A\to
  C$ is also $s_c$-subintegral.
\end{enumerate}
\end{prop}

\begin{proof} Notice that the rings $A$, $B$, $C$ and $A'$ satisfy
  \hyp\, and have isomorphic total rings of fractions \cite[Lem. 2.8]{FMQ-futur}. Assume $A\to
  B$ is $w_c$-subintegral and let $\p\in\RCent A\cap \MSp A$. By
  Proposition \ref{finitebirationalcentralintclos} we know that the maps
  $\RCent C\cap \MSp C\to\RCent A\cap \MSp A$ and $\RCent B\cap \MSp
  B\to\RCent C\cap \MSp C$ are surjective and moreover $\RCent B\cap
  \MSp B\to\RCent A\cap \MSp A$ is bijective ($A\to
  B$ is $w_c$-subintegral). It follows that there exists a unique central
  maximal ideal of $C$ lying over $\p$ and this ideal is $\q\cap C$ where
  $\q$ is the unique central maximal ideal of $B$ lying over $\p$. We
  get a sequence of finite extensions of fields $k(\p)\to k(\q\cap
  C)\to k(\q)$ and since $k(\p)$ and $k(\q)$ are isomorphic then the
  proof of (1)
  is done. The proof of (2) is similar replacing maximal ideals by
  prime ideals.
\end{proof}


Here is the central version of weak-normalization:
\begin{defn} Let $A$ be a ring which has an integral closure denoted by
$A'$. The ring $$A^{w_c}=\{f\in A'|\, \forall\m\in\RCent
A \cap \MSp A,\,f_{\m}\in
A_{\m}+\JRadCe(A'_{\m})\}$$ 
is called the weak-normalization of $A$ relative to its central locus, or $w_c$-normalization for short. In case $A=A^{w_c}$, we say that $A$ is
centrally weakly-normal.
\end{defn}


The weak-normalization relative to the central locus satisfies a universal property, that will be an important step for the study of the geometric $w_c$-normalization in the sequel.
\begin{prop}\label{UniversalWCrings}
  Let $A$ be a ring satisfying \hyp\, which has an integral closure denoted by
$A'$. For an extension $A\rightarrow B$ which injects into $A'$ the
following properties are equivalent:
\begin{enumerate}
 \item[(i)] The extension $A\to B$ is $w_c$-subintegral.
 \item[(ii)] $B\subset A^{w_c}$.
\end{enumerate}
\end{prop}

\begin{proof}
We begin by showing that $A\to A^{w_c}$ is $w_c$-subintegral. Let us mention first that, according to Proposition \ref{finitebirationalcentralintclos}, 
one has a canonical surjection from 
$\RCent A^{w_c}\cap \Max A^{w_c}$ onto $\RCent A\cap \Max A$. To show the injectivity, let us consider $\q_1$ and $\q_2$
in $\RCent A^{w_c}\cap \Max A^{w_c}$ lying over $\p\in\RSp A\cap \Max A$.
Using again Proposition \ref{finitebirationalcentral}, one has two central maximal ideals
$\ir_1$ and $\ir_2$ in $\RCent A'\cap\Max A'$ lying over $\q_1$ and $\q_2$ respectively.
Let $f\in \q_1$, 
then $f_ {\p}$ $=\alpha+\beta$ where $\alpha\in A_ {\p}$ and $\beta\in\JRadCe{A'_{\p}}$.
Then, $\beta \in \ir_1\cap \ir_2$ and hence $\beta\in \q_2(A^{w_c})_{\p}$ and $\alpha\in\p A_{\p}$. 
We get then $f\in \q_2(A^{w_c})_{\m}$
for any maximal ideal $\m$ in $A$ 
(if $\m\not=\p$, one has $\q_2(A^{w_c})_{\m}=(A^{w_c})_{\m}$).
This shows (using for instance \cite[Thm. 4.6]{Ma}) that $f\in \q_2$ and hence 
$\q_1\subset\q_2$. By symmetry $\q_1=\q_2$.

It remains to show that $A^{w_c}$ is equiresidual over $A$ at any
central maximal ideal $\m$ in $A$.
Let us consider the ideal 
$\m^{w_c}=(\JRadCe{A'_{\m}})\cap (A^{w_c})_{\m}$. 
Using the natural identification between 
the prime ideals in $A'$ lying over $\m$ and
the maximal ideals of ${A'}_{\m}$ lying over $\m A_{\m}$,
it is clear that, over $\m^{w_c}$ in $A'_{\m}$
lie all the central maximal ideals  of 
$A'_{\m}$ and hence $\m^{w_c}$ is maximal and central.
Moreover $\m^{w_c}$ lies over $\m A_{\m}$
Then, $\m^{w_c}$ is the only
central real maximal ideal of $(A^{w_c})_{\m}$ lying over $\m A_{\m}$. Moreover, one 
clearly has $A/\m\simeq (A^{w_c})_{\m}/\m^{w_c}$ and thus $A\to A^{w_c}$ is $w_c$-subintegral.

Consider an extension $A\rightarrow B$ which injects into $A'$. If
$B\subset A^{w_c}$ then $A\to B$ is $w_c$-subintegral by 1) of Proposition
\ref{intermedwc}. It follows that (ii) implies (i).

Let us show now that (i) implies (ii). Assume $A\to B$ is
$w_c$-subintegral and $B\subset A'$. 
By definition of $A^{w_c}$, one has to show that $B_{\p}\subset A_{\p}+\JRadCe{A'_{\p}}$ 
for any maximal $\p$ in $\RCent A$.
Let us denote by $\q$ the unique 
maximal ideal in $\RCent B$ lying over $\p$. 
By Proposition
\ref{finitebirationalcentralintclos}, one has $\q A'_{\p}\subset \JRadCe{A'_{\p}}$ which shows that $\q B_{\p}\subset
\JRadCe{A'_{\p}}$. We use the following commutative diagram
\[
\begin{array}{ccc}
k(\p) & \simeq & k(\q)  \\
\uparrow &&\uparrow\\
A_{\p} &\rightarrow & B_{\p}
\end{array}
\]\\
Let $b\in B$. There exist $a\in A$ and $s\in A\setminus \p$ such that
$a/s=b$ in $k(\q)$. Hence $a-sb\in\q$ and thus $b-a/s\in \q
B_{\p}=\ker (B_{\p}\to k(\q))$. We get $b\in A_{\p}+\JRadCe{A'_{\p}}$.
This concludes the proof.
\end{proof}


Using Propositions \ref{subintegralcomposition} and \ref{UniversalWCrings}, one readily derives 
an idempotency property:
\begin{prop}\label{WeakIdempotency}
For any ring $A$ satisfying \hyp, the ring $A^{w_c}$ is centrally weakly normal.
\end{prop}

We introduce now the concept of seminormalization relative to the central locus.

\begin{defn}\label{defrealseminormal}
Let $A$ be a ring which has an integral closure denoted by
$A'$. The ring $$A^{s_c}=\{f\in A'|\, \forall\p\in\RCent
A,\,f_{\p}\in
A_{\p}+\JRadCe{A'_{\p}}\}$$ 
is
called the seminormalization of $A$ relative to its central locus, or $s_c$-normalization for short. In case $A=A^{s_c}$, we say that $A$ is
centrally seminormal.
\end{defn}

The seminormalization relative to the central locus satisfies a universal property in the same spirit as the weak normalization to the central locus,
and the proof is similar as that of Proposition \ref{UniversalWCrings} (it suffices to replace
maximal ideals with prime ideals). 
It will be an important result for the geometric $s_c$-normalization in the sequel.
\begin{prop}\label{UniversalNCrings}
 Let $A$ be a ring satisfying \hyp\, which has an integral closure denoted by
$A'$. For an extension $A\rightarrow B$ which injects into $A'$ the
following properties are equivalent:
\begin{enumerate}
 \item[(i)] The extension $A\to B$ is $s_c$-subintegral.
 \item[(ii)] $B\subset A^{s_c}$.
\end{enumerate}
\end{prop}

\begin{rem}
	We have a sequence of inclusions
$$A\subset {}^+A \subset  A^{s_c} \subset A^{w_c} \subset
A',$$
so that a centrally weakly-normal ring is automatically centrally seminormal, and similarly a centrally
seminormal ring is automatically seminormal. 
\end{rem}

Using Propositions \ref{subintegralcomposition}, \ref{WeakIdempotency} and \ref{UniversalNCrings}, one readily gets 
an idempotency property:
\begin{prop}\label{SemiIdempotency}
For any ring $A$ satisfying \hyp\, one has :
\begin{enumerate}
 \item $A^{s_c}$ is centrally seminormal,
 \item $(A^{w_c})^{s_c}=A^{w_c}$,
 \item $(A^{s_c})^{w_c}=A^{w_c}$.
\end{enumerate}
\end{prop}

We end this section with a comparison with the biregular integral closure that has been defined in \cite{FMQ-futur}, in relation with the notion of biregular extension recalled in section \ref{sect-CSE}. 

Let $A$ be a ring satisfying \hyp. The biregular integral closure
$A^b$ of $A$ is the integral closure of
$A$ in $\SO(A)$, namely $A^b=A'_{\SO(A)}$. By
\cite[Thm. 4.12]{FMQ-futur}, $A^b$ is the biggest ring contained in
$A'$ which is a biregular extension of $A$.
Since biregular extensions are $s_c$-subintegral by Corollary
\ref{implbireg}, we have a sequence of inclusions
$$A\subset A^b \subset  A^{s_c} \subset A^{w_c} \subset A'.$$ 
It follows also that a centrally
seminormal ring (which is automatically seminormal) is equal to its biregular integral closure.




\section{Finite birational maps and continuous functions}\label{finite-birat}

This section is devoted to a topological study of finite birational maps in real geometry. We begin with considering Euclidean topology, and the restriction of the maps to the central loci. We study later the action of finite birational maps on continuous rational functions, respectively hereditarily continuous rational functions, on the central loci. The main result is Theorem \ref{bijfinitebiratbis} (and Theorem \ref{bijfinitebiratbis-reg} for the hereditary case) which gives necessary and sufficient conditions on the restriction to the central loci to be an homeomorphism.

\subsection{Properties of finite birational maps}

A birational map $\pi:Y\rightarrow X$ between
algebraic sets is a polynomial
map that induces an isomorphism from a Zariski-dense open subset of
$Y$ to a Zariski-dense open subset of $X$. Equivalently, $\pi$ is birational if and only if $\pi$ induces an
isomorphism $\K(X)\simeq \K(Y)$ at the level of the total rings of fractions.
Such a map $\pi$ is
defined everywhere, but its inverse may have indeterminacy points.  
Recall that $\pi$ is finite
if the associated ring morphism $\Pol(X)\rightarrow \Pol(Y)$ makes
$\Pol(Y)$ a finite $\Pol(X)$-module, a property which implies that the 
ring extension $\Pol(X)\rightarrow \Pol(Y)$ is integral.

\begin{lem} 
\label{closedeucl}
Let $\pi:Y\rightarrow X$ be a birational map between
algebraic sets. The induced ring morphism $\Pol(X)\rightarrow \Pol(Y)$ is injective. If moreover $\pi$ is finite, then the map $\pi$ is proper (and hence closed) for the Euclidean topology.
\end{lem}

\begin{proof}
The ring morphism $\Pol(X)\rightarrow \Pol(Y)$ is injective since
$\pi$ is birational and thus dominant.

Assume moreover that $\pi$ is a finite birational map. We show that the map $\pi$ is closed and proper with respect to the real spectrum topology, then with respect to
the semi-algebraic topology and finally with respect to the Euclidean topology.
By \cite[Prop 4.2-4.3]{ABR}, the induced map
$\Sp_r\Pol(Y)\rightarrow \Sp_r\Pol(X)$ is closed for the real
spectrum topology. According to \cite[Theorem 7.2.3]{BCR}, there is a 
bijective correspondence between open (resp. closed) semi-algebraic subsets of 
$X$ (resp. $Y$) and open (resp. closed) constructible subsets of the real spectrum $\Sp_r \Pol(X)$
(resp. the real spectrum $\Sp_r \Pol(Y)$). It follows that the image
by $\pi$ of every closed semi-algebraic subset of $Y$ is a closed
semi-algebraic subset of $X$. 
Now it is classical (\cite{vdd} for instance) to conclude 
that $\pi$ is closed for the Euclidean topology.

The morphism $\pi$ being finite, it has compact fibers
and hence is proper.
\end{proof}

We will often consider intermediate algebraic sets between a given algebraic set and its normalization, in the sense of the next statement. 
Before recall that, given 
a ring extension $A_1\rightarrow A_2$ (an injective ring
homomorphism), the ring $A$ is said to be an 
intermediate ring between $A_1$ and $A_2$ if  one has a factorization $A_1\rightarrow A\rightarrow A_2$ 
where the morphisms are extension rings. 

\begin{prop}\label{intermediatering}
Let $X$ be an algebraic set and $\pi:Y\rightarrow X$ be a finite birational map between
algebraic sets. Let $A$ be an intermediate ring
between $\Pol(X)$ and $\Pol(Y)$. There exists a unique algebraic set $Z$ such that
$A=\Pol(Z)$. Moreover the induced maps $Y\rightarrow Z$ and
$Z\rightarrow X$ are finite and birational.
\end{prop}

When the conditions of Proposition \ref{intermediatering} are satisfied, we say that $Z$ is an
intermediate algebraic set between $X$ and $Y$.

\begin{proof}
Note that $A$ is integral over $\Pol(X)$ since $\Pol(X)\hookrightarrow \Pol(Y)$ is an integral morphism. Moreover $\Pol(X)\subset A\subset
\K(X)$ since $\pi$ is birational and $A$ is integral over $\Pol(X)$
therefore $A\subset \Pol(X')$. In particular $A$ has $\K(X)$ as total
ring of fractions \cite[Lem. 2.8]{FMQ-futur}.
First note that $\Pol(X')$ is a finite $\Pol(X)$-module (this property is known to be satisfied for the so-called Mori rings, a sup-class of finite-type $\R$-algebras). 
Since $A$
is integral over $\Pol(X)$ then $A$ is a submodule of the Noetherian
$\Pol(X)$-module $\Pol(X')$ so that $A$ is
a finite $\Pol(X)$-module. As a consequence $A$ is a finitely
generated $\R$-algebra, so that there exists an ideal $I$ of some
$\R[X_1,\ldots,X_m]$ such that $A$ is isomorphic to
$\R[X_1,\ldots,X_m]/I$. Note that $I$ is a real ideal since the total
ring of fractions of $A$ is $\K(X)$, so that $A$ is the ring of polynomial
functions of an algebraic set $Z$ by the real Nullstellensatz (see \ref{sectiongeoalg}). The induced maps  $Y\rightarrow Z$ and
$Z\rightarrow X$ are clearly finite by the above arguments.
\end{proof} 

As already mentioned a finite birational map is not necessarily surjective in general. This is the motivation to restrict our attention to the central loci. 

\begin{prop}
\label{normalizationcentral1bis}
Let $\pi:Y\rightarrow X$ be a finite birational map between
algebraic sets. Then
\begin{enumerate}
\item $\pi:\Cent Y\rightarrow\Cent X$ is well defined and surjective.
\item $\pi:\Cent Y\rightarrow\Cent X$ is a quotient map for the
  Euclidean topology.
\item The
composition by $\pi$ gives an isomorphism between the ring $\SR(\Cent X)$ and the subring of functions in 
$\SR(\Cent Y)$ 
that are constant on the fibers of $\pi:\Cent Y\rightarrow\Cent X$.
\end{enumerate}
\end{prop}

\begin{proof}
The fact that $\pi$ maps surjectively $\Cent Y$ onto $\Cent X$ is given by Proposition \ref{finitebirationalcentral}.

By (1) and Lemma \ref{closedeucl} the map $\pi$ is continuous,
surjective and a closed map for the Euclidean topology; this gives
(2). It means that a map is continuous on $\Cent X$ if and only if the
composition by $\pi$ is continuous on $\Cent Y$. Consequently, for a
continuous function $f:\Cent Y\to \R$ constant on the fibers of $\pi$
then there exists a (unique) continuous function $g:\Cent X\to \R$ such that $f=g\circ\pi$.
Using moreover that $\pi$ is birational, we get (3).
\end{proof}

Note that we cannot replace continuous rational functions by continuous hereditarily rational functions in the third point as illustrated by Koll\'ar surface. Note moreover that if $X$ is central, this point says that the
composition by $\pi$ gives a injective ring morphism $\pi_0:\SR(X)\hookrightarrow \SR(Y)$ whose image is the subring of
functions in 
$\SR(Y)$ that are constant on the fibers of $\pi$. 

\vskip 2mm

Given an algebraic set $X$, we investigate now the action of adding to $\Pol(X)$ a continuous rational function $f$ which is integral over $\Pol(X)$. Forgetting first about the continuity, this is done as follows.
The canonical morphism $\Pol(X)\rightarrow \K(X)$ factorizes through the morphism $\phi:\Pol(X)[t]\rightarrow\K(X)$ 
defined by $t\mapsto f$, inducing a morphism $\Pol(X)[t]/\Ker \phi\rightarrow\K(X)$.
Since $f$ is integral, the ring homomorphism $\Pol(X)\rightarrow \Pol(X)[t]/\Ker \phi$ is finite and  
$\Pol(X)[t]/\Ker \phi$ is the coordinate ring $\Pol(Y)$ of an algebraic set $Y$. 
In this setting, $f$ corresponds to the new variable $t$.

Taking into account the continuity, we obtain the following. 

\begin{prop}
\label{comparaisonfonction}
Let $X$ be an algebraic set. Let $f$ be a rational
function on $X$ that is integral on $\Pol(X)$ and assume that $f\in\SR(\Cent X)$. 
Denote by $Y$ the algebraic
set such that $\Pol(Y)=\Pol(X)[f]$, $t$ the polynomial function in $\Pol(Y)$ that corresponds to $f$
and $\pi:Y\to X$ the associated
finite birational map. Then the continuous rational function $f\circ\pi$ coincides with the polynomial function $t$ on $\Cent Y$. 
\end{prop}

\begin{proof}
The function $t$ is a polynomial extension to $Y$ of
$f_{|U}\circ\pi_{|\pi^{-1}(U)}$ where $(f_{|U},U)$ is a regular
presentation of $f$. It follows that $f\circ\pi=t$ on a 
Zariski-dense open subset of $Y$ and we conclude by density with respect to Euclidean topology that both functions coincide on $\Cent Y$. 
\end{proof}

Note that in general $f\circ\pi$ does not coincide with $t$ on the
whole of $Y$ even if $f$ is continuous rational on the whole $X$. Consider for instance the curve $X$ given by
$\Z(y^4-x(x^2+y^2))$ as in Example \ref{grospoint2}. The rational function $f=y^2/x$ satisfies the integral equation $f^2-f-x=0$
and $\Pol(Y)=\Pol(X)[y^2/x]$.
Denoting $\pi:Y\to X$, the rational continuous function $f\circ \pi$ is not equal to $t$ on whole
$Y$ but only on $\Cent Y$. The subsets $\Cent Y$ and $\Cent X$ are in bijection but it is not the case for $X$ and $Y$.

\vskip 2mm

In general, the issue whether a given continuous rational function may be lifted to a polynomial one via a finite birational map is crucial in our discussion. Next lemma illustrates two situations that will be useful in the sequel.
 
\begin{lem}
\label{fondamental}
Let $X$ be an algebraic set with normalization
$\pi':X'\to X$. Let $Y$ be an algebraic set such that
there exist finite birational maps $\pi:Y\to X$ and $\varphi:X'\to Y$
satisfying $\pi'=\pi\circ\varphi$. Then
\begin{enumerate}
\item Let $f\in\Pol(Y)$ and $\tilde{f}\in\SR(\Cent X)$. Then $f=\tilde{f}\circ \pi$ on
  $\Cent Y$ if and only if $f\circ\varphi=\tilde{f}\circ\pi'$ on
  $\Cent X'$.
\item Assume $Y$ is central. Let $f\in\Pol(Y)$ and $\tilde{f}\in\SR(X)$.
Then $f=\tilde{f}\circ \pi$ on
  $Y$ if and only if $f\circ\varphi=\tilde{f}\circ\pi'$ on
  $X'$.
\end{enumerate}
\end{lem}

\begin{proof}
The proof follows from the surjectivity of the maps $\Cent X'\to\Cent X$, $\Cent X'\to\Cent Y$,
$\Cent Y\to\Cent X$, and also $X'\to Y$ in the case $Y$ is central,
given by Proposition \ref{normalizationcentral1bis}.
\end{proof}

\subsection{Finite birational bijection on the central locus}

The normalization process allows to simplify the singularities of an
algebraic set $X$ by adding integral rational functions to the
polynomial ring of $X$. Requiring an additional continuity of the
integral functions leads to simplify the singularities without loosing
some bijectivity property with some part of $X$.




This is the main result of this section.

\begin{thm}
\label{bijfinitebiratbis}
Let $\pi:Y\rightarrow X$ be a finite birational map between
algebraic sets, and denote by $\pi_{|\Cent Y}:\Cent Y\rightarrow\Cent X$ its restriction to the central loci. 
The following properties are equivalent:
\begin{enumerate}
\item[(i)] $\pi_{|\Cent Y}$ is a bijection.
  \item[(ii)] The ring morphism
  $\pi_0:\SR(\Cent X)\rightarrow \SR(\Cent Y)$ is an isomorphism.
  \item[(iii)] 
 For all $g\in\Pol(Y)$ there exists
   $f$ in $\SR(\Cent X)$
   such
   that $g=f\circ\pi$ on $\Cent Y$.
\item[(iv)] $\pi_{|\Cent Y}$ is an homeomorphism for the Euclidean topology.
\item[(v)] The morphism $\Pol(X)\rightarrow \Pol(Y)$ is centrally
  weakly subintegral.
\item[(vi)] The rational morphism $\pi^{-1}$ admits a (rational) continuous extension to $\Cent X$.
\end{enumerate}
\end{thm} 

Note that in (vi), the restriction of $\pi^{-1}$ to $\Cent X$ is not defined entirely on $\Cent X$, but only on a subset dense in $\Cent X$ with respect to Euclidean topology. So that (vi) means precisely that there exists a continuous map on $\Cent X$ which admits $\pi^{-1}$ as a rational model.

\begin{proof} Let's prove first the equivalence between (i) and
  (ii). The fact that (i) implies (ii) is a direct consequence of 3)
  of Proposition \ref{normalizationcentral1bis}. To prove the converse implication, assume that $\pi_0:\SR(\Cent X)\rightarrow \SR(\Cent Y)$ is an isomorphism whereas $\pi_{|\Cent Y}$ is
not bijective. There exists $x\in \Cent X$ such that we have $\{ y_1,y_2\}\subset
\pi^{-1}(x) \cap \Cent Y$ and $y_1\not=y_2$. There exists $p\in \Pol(Y)$
such that $p(y_1)\not=p(y_2)$. By 3) of Proposition
\ref{normalizationcentral1bis}, we get that $p\in\SR(\Cent Y)\setminus\pi_0(\SR(\Cent X))$
since $p$ is not constant on the fibers of $\pi_{|\Cent Y}$, a contradiction. We
have proved that (ii) implies (i).

Since (ii) implies clearly (iii) then (i) implies (iii).
To show that (iii) implies (i) it suffices to see that for any $y_1,y_2$ in $\Cent Y$ such that $\pi(y_1)=\pi(y_2)=x$
and for any $g\in \Pol(Y)$ such that $g(y_1)=0$, on has also $g(y_2)=0$.
Since there exists $f$ in $\SR(\Cent X)$
   such
   that $g=f\circ\pi$ on $\Cent Y$, on gets the result.

Note that (iv) implies trivially (i), whereas (i) implies (iv) since $\pi$ is closed with respect to Euclidean topology by Lemma \ref{closedeucl}.

We clearly have that (v) and (i) are equivalent since $\RCent
\Pol(X)\cap\Max \Pol(X)=\Cent X$ and $\RCent
\Pol(Y)\cap\Max \Pol(Y)=\Cent Y$.

Note that (vi) implies (i) since the rational inverse $\pi^{-1}$ admits
a continuous extension along $\Cent X$, which is indeed an inverse for
$\pi_{|\Cent Y}$ by continuity with respect to Euclidean topology. To show the converse, we need to prove
that $\pi_{|\Cent Y}^{-1}$ is a rational continuous map. Consider $Y\subset \R^n$ and choose a coordinate
function $y_i$ on $Y$ for $i \in \{1,\ldots,n\}$. We want to prove
that the rational function $z_i=y_i \circ \pi_{|\Cent Y}^{-1}$ is continuous
on $\Cent X$. However $z_i\circ \pi_{|\Cent Y}$ is polynomial on
$\Cent Y$, 
so that, by (ii), $z_i$ belongs to $\SR(\Cent X)$ and thus
$\pi_{|\Cent Y}^{-1}:\Cent X \to \Cent Y$ is a rational
continuous map. 
\end{proof}

The centrality of an algebraic set is not preserved under finite birational
maps in general. However it will be the case if we assume moreover the map to be bijective. 
The following result enumerates the properties of a bijective finite
birational map onto a central algebraic set. 
\begin{prop}
\label{bijfinitebirat}
Let $\pi:Y\rightarrow X$ be a finite birational map between
algebraic sets where $X$ is
central. Let us assume that $\pi$ is a bijection. Then, one has the following properties:
\begin{enumerate}
\item $Y$ is central. 
\item The canonical morphism $\SR(X)\rightarrow \SR(Y)$ is an isomorphism.
\item $\pi$ is an homeomorphism for the constructible topology.
\item The morphism $\Pol(X)\rightarrow \Pol(Y)$ is centrally weakly subintegral.
\item $\pi^{-1}$ is rational continuous.
\end{enumerate}
\end{prop}

\begin{proof} 
Since $X$ is assumed to be central, we know by Proposition
\ref{normalizationcentral1bis} that $\pi_{|\Cent Y}$ is surjective
onto $X$. In particular, if $\pi$ is assumed to be bijective, then $Y$ is automatically central. 

Note that (2), (4) and (5) are direct consequences of Theorem \ref{bijfinitebiratbis}.

Let us show (3). By Lemma \ref{closedeucl}, $\pi$ is closed for the Euclidean topology so that using
\cite[Cor. 4.9]{KP}, the image by $\pi$ of a Zariski constructible closed subset of
$Y$ is a Zariski constructible closed subset of $X$. It follows that $\pi$ is an
homeomorphism for the constructible topology.
\end{proof}

\begin{rem}
 Note that a bijective finite birational polynomial map onto a central
algebraic set
  is not necessarily an isomorphism while it is an isomorphism in the
  category of rational continuous maps by property 5). For instance, let $X$ be the
  cuspidal curve given by $y^2=x^3$ in $\R^2$, and $X'$ be its
  normalization. The normalization map $\pi:X'\rightarrow X$ is
  birational, finite and bijective. It is even an homeomorphism with respect to the Zariski topology (the curves are irreducible, so the Zariski subsets are just points). However $X$ is singular whereas $X'$ is smooth.
\end{rem}

    

\subsection{Finite hereditarily birational bijection on the central locus}
\label{heredfinitebirat}
In this subsection we restrict our attention to continuous hereditarily rational functions.

Recall that the restriction of a rational continuous functions does not remain rational in general.
This phenomenon appears also for birational maps.
Namely, let $\pi:Y\rightarrow X$ be a finite birational map between
algebraic sets. By Lemma \ref{closedeucl} the corresponding
morphism $\Pol(X)\rightarrow \Pol(Y)$ is injective and integral.
Let $W$ be an irreducible algebraic subset
of $Y$. There exists $\q\in \ReSp \Pol(Y)$ such that $W=\Z(\q)$. We
denote by $\p$ the real prime ideal $\q\cap \Pol(X)$ and by $V$ the
real irreducible algebraic subset of $X$ given by $\Z(\p)$. The
restriction of $\pi$ to $W$ gives clearly a map
$\pi_{|W}:W\rightarrow V$ which is finite since the corresponding
morphism of polynomial functions 
$$\dfrac{\Pol(X)}{\p}\rightarrow\dfrac{\Pol(Y)}{\q}$$
is integral. 
The residue field $k(\q)=\K(W)$ is an algebraic extension of $k(\p)=\K(V)$. 
Then, $\pi_{|W}$ remains birational if and only if $\pi$ induces an isomorphism $k({\p})\simeq
k({\q})$. 

\begin{ex}\label{K-surf}
Consider Koll\'ar surface
$X=\Z(x^3-y^3(1+z^2))$. Its normalization is given by
  $\Pol(X')=\Pol(X)[x/y]=\Pol(X)/(t^3-(1+z^2),yt-x)\simeq \R[t,y,z]/(t^3-(1+z^2))$, setting
  $t=x/y$. 
Let ${\p}=(x,y)\in \ReSp
\Pol(X)$ and let ${\q}\in\ReSp \Pol(X')$ be the unique real
prime ideal of $\Pol(X')$ such that ${\q}\cap
\Pol(X)={\p}$. We have $k({\p})=\R(z)$ and 
$k({\q})=\R(z)(^3\sqrt{1+z^2})\not\simeq k({\p})$. Here, the normalization map $\pi':X'\rightarrow
  X$ $(y,t,z)\mapsto (ty,y,z)$ is a bijective finite birational map
  which is not hereditarily birational in the sense that $\pi'$ is
  birational but $\pi'_{|\Z(\q)}:\Z(\q)\to\Z(\p)$ is not
  birational. The map $\pi'$ has an inverse bijection $\pi'^{-1}$
  given by $(x,y,z)\mapsto (y,x/y,z)$ if $y\not=0$ (the inverse of
  $\pi'$ as a birational map) and $(0,0,z)\mapsto
  (0,^3\sqrt{1+z^2},z)$. We see that $\pi'^{-1}$ is rational
  continuous but not hereditarily rational.
  Moreover, $\Pol(X)\hookrightarrow \Pol(X')$ 
is $w_c$-subintegral but not $s_c$-subintegral and one may also notice that 
$\Pol(X)\hookrightarrow \Pol(X)[x^2/y]$ is $s_c$-subintegral.
\end{ex}

This consideration leads us to the definition of hereditarily birational maps on the central locus, inspired by \cite{KN}.
\begin{defn}
\label{defrestriction}
Let $\pi:Y\rightarrow X$ be a birational map between
algebraic sets. We say that $\pi$ is hereditarily birational on the central locus of $Y$ if for every irreducible algebraic subset $W$ central in $Y$, 
the restriction $\pi_{|W}$ is birational with the Zariski closure of its image.
\end{defn}

Example \ref{K-surf} provides a finite birational bijection between central
algebraic sets such that the natural map $\RCent
\Pol(Y) \to \RCent \Pol(X)$ is a bijection, but which is not hereditarily birational.

Note moreover that a finite birational map $\pi:Y\rightarrow X$ such that $\Pol(X)\rightarrow \Pol(Y)$ is centrally weakly subintegral and hereditarily birational on the central locus of $Y$ does not necessarily give rise to a centrally subintegral extension :

\begin{ex}\label{Ex-primepasbij} Consider the surface $X=\Z(yz^4-(y^2-x^4)^4))$ in $\R^3$. The rational function $f=(\frac{y^2-x^4}{z})^2$ is integral over $\Pol(X)$ and admits a continuous extension to $\Cent X$ by $+\sqrt y$. The algebraic set obtained by adding $f$ to $\Pol(X)$ as in Proposition \ref{comparaisonfonction} is isomorphic to $Y=\Z(yz^2-(y^4-x^4)^2)\subset\R^3$. The restriction of the induced finite birational map $\pi :Y\to X$ to the central loci is a bijection, so that the extension $\Pol(X) \to \Pol(Y)$ is $w_c$-subintegral.

Consider $V\subset X$ the parabola defined by $z=0$ and $y=x^2$. Then $V\subset \Cent X$, and there exist two irreducible curves central in $Y$ and lying over $V$ : the lines $W_\pm$ given by $z=0$ and $y=\pm x$. In particular the extension $\Pol(X) \to \Pol(Y)$ is not centrally subintegral.
However both of these lines are isomorphic to $V$ via the restriction
of $\pi$, so that $\pi$ is hereditarily birational by restriction to
$W_+$ and $W_-$ and moreover $\pi$ is hereditarily birational on the
central locus of $Y$ since $W_+\cup W_-$ is the inverse image by $\pi$
of the indeterminacy locus of $f$.

Note moreover that just half of each line is included in $\Cent Y$,
and that the restriction of $\pi^{-1}$ to $V$ is given by the
semialgebraic function $x\mapsto (x,|x|,0)$. It shows that the inverse
bijection $(\pi_{|\Cent Y})^{-1}:\Cent X\to \Cent Y$ is rational
continuous but not hereditarily rational.
\end{ex}

However, we readily get :

\begin{lem}\label{lem-her} Let $\pi:Y\rightarrow X$ be a finite birational map between
algebraic sets. 
The morphism $\Pol(X)\rightarrow \Pol(Y)$ is
   centrally subintegral if and only if $\pi$ is hereditarily birational on the central locus of $Y$ and the map $\RCent
\Pol(Y) \to \RCent \Pol(X)$ is a bijection.
\end{lem}

We aim to characterize the central subintegrality of a finite
birational map by hereditary properties of the inverse birational
map. More precisely, we are going to consider birational maps of the form $\pi:Y\rightarrow X$ such that $\pi^{-1}$ is hereditarily rational on the center of $X$. This property is strictly stronger than $\pi^{-1}$ being only rational continuous as noticed in Examples \ref{K-surf} and \ref{Ex-primepasbij}.
Beware moreover that a birational map bijective between the central loci and hereditarily birational does not necessarily satisfy that $\pi^{-1}$ is hereditarily rational on the center of $X$ as illustrated by Example \ref{Ex-primepasbij}.

The subtle differences between hereditarily birational maps and rational maps satisfying that $\pi^{-1}$ is hereditarily rational is crucial in this section.

\begin{prop}
\label{prop-her}
Let $\pi:Y\rightarrow X$ be a finite birational map between
algebraic sets. 
Then $\pi^{-1}$ is hereditarily rational on the center of $X$ if and only if  $\pi$ is hereditarily birational on the central locus of $Y$ and $\RCent
\Pol(Y) \to \RCent \Pol(X)$ is a bijection.
\end{prop}

\begin{proof}
Assume $\pi^{-1}$ is hereditarily rational on the center of $X$. Let $W$ be an irreducible subset central in $Y$, and denote by $V\subset X$ the Zariski closure of $\pi(W)$. 
Note that $\pi_{|W}$ is still a finite polynomial map, and that $V$ is central in $X$. 

The restriction of $(\pi_{|\Cent Y})^{-1}$ to $V\cap \Cent X$
coincides on $V\cap \Cent X$ with the continuous extension of a
rational map $\phi$ on $V$ by assumption on $\pi$. In particular
$\phi$ is a regular inverse of $\pi_{|W}$ on a Zariski dense open subset
$V^o\subset V$ intersected with $\Cent X$. Since $V$ is central in
$X$, it implies that $\phi$ is a rational inverse for
$\pi_{|W}$. Therefore $\pi$ is hereditarily birational on $\Cent Y$.

For $V$ irreducible and central in $X$, let $W_i$ be the central
subsets of $Y$ lying over $V$ for $i=1,...,r$, all equipped with a
rational inverse $\phi_i :V \to W_i$ of $\pi_{|W_i}$ (we have already
shown that $\pi$ is hereditarily birational on $\Cent Y$). However
these inverses coincide with $(\pi_{|\Cent Y})^{-1}$ on a Zariski
dense open subset of $V$ intersected with $V\cap \Cent X$, so that they coincide with a common rational map $\phi$. In particular there exists a semialgebraic subset $S$ of $V\cap \Cent X$ of maximal dimension on which $\phi$ is injective, and $\phi(S)\subset W_i$ for $i=1,\ldots,r$. As a consequence $W_1=\cdots=W_r$ by centrality of the $W_i$'s in $Y$. We have proved that $\RCent
\Pol(Y) \to \RCent \Pol(X)$ is a bijection.

\vskip 2mm

Conversely, let $V$ be central in $X$ and denote by $W$ the unique central set in $Y$ lying over $V$. By assumption $\pi_{|W}:W\to V$ is birational, and we aim to prove that its inverse is a rational model for the restriction of $(\pi_{|\Cent Y})^{-1}$ to $V\cap \Cent X$.

By Lemma \ref{lemm1}, there exists $Z\subset V$ with $\dim Z<\dim V$ such that the inverse image by $\pi$ of $(V\setminus Z)\cap \Cent X$ is included in $W \cap \Cent Y$. In other terms
$$  (V\setminus Z)\cap \Cent X \subset \pi(W\cap \Cent Y),$$
so that the restriction to $V\cap \Cent X$ of the continuous map
$(\pi_{|\Cent Y})^{-1}$ coincides with the rational map
$(\pi_{|W})^{-1}$ on the intersection of $V\cap \Cent X$ with a
Zariski dense open subset of $V$. It means that $\pi^{-1}$ is hereditarily rational on the central locus of $X$.
\end{proof}

To emphasize a crucial point in the preceding proof, note that the very last argument fails in Example \ref{Ex-primepasbij} since for both $W$ over $V$, the image $\pi(W\cap \Cent Y)$ does not contains generically $V\cap \Cent X=V$, but only half of it.

\vskip 2mm

Note that a finite birational map $\pi:Y\rightarrow X$ induces a morphism $\pi^0:\SRR(\Cent X)\rightarrow \SRR(\Cent Y)$ given by composition
  with $\pi$. With this in mind, we characterize centrally subintegral extensions using hereditarily rational functions. This provides an analog of Theorem \ref{bijfinitebiratbis} in the hereditary context.

\begin{thm}\label{bijfinitebiratbis-reg} Let $\pi:Y\rightarrow X$ be a finite birational map between
algebraic sets. 
The following properties are equivalent:
\begin{enumerate}
 \item[(i)] The ring morphism
  $\pi^0:\SRR(\Cent X)\rightarrow \SRR(\Cent Y)$ is an isomorphism.
 \item[(ii)] 
 For all $g\in\Pol(Y)$ there exists an hereditarily rational function
   $f$ in $\SRR(\Cent X)$
   such
   that $g=f\circ\pi$ on $\Cent Y$.
 \item[(iii)] The rational map $\pi^{-1}$ is hereditarily rational on the center of $X$.
 \item[(iv)] The extension $\Pol(X)\to \Pol(Y)$ is centrally subintegral.
 \end{enumerate}
\end{thm}

\begin{proof} Note first that $(iii)$ and $(iv)$ are equivalent by Lemma \ref{lem-her} and Proposition \ref{prop-her}. 
It is immediate that $(i)$ implies $(ii)$ since $\Pol(Y) \subset \SRR(\Cent Y)$.

Assume $(ii)$. Note first that $\pi_{|\Cent Y}$ is bijective by
Theorem \ref{bijfinitebiratbis}. Assume moreover $Y\subset \R^n$ and
consider a coordinate function $y_i$ on $Y$. By assumption there
exists $f$ in $\SRR(\Cent X)$ such that $y_i=f\circ\pi$ on $\Cent
Y$. This implies that the function $y_i \circ (\pi_{|\Cent Y})^{-1}$
is an hereditarily rational function on $\Cent X$. In particular we have $(iii)$.

Assume $(iii)$. The proof of $(i)$ comes down to prove the surjectivity of $\pi^0$. So for $g\in \SRR(\Cent Y)$, we have to prove that $f\in \SR(\Cent X)$ given by $f=g \circ (\pi_{|\Cent Y})^{-1}$ is in fact in $\SRR(\Cent X)$. So let $V$ be central in $X$. Using Proposition \ref{prop-her}, there exists a unique central set $W$ in $Y$ lying over $V$, and $\pi_{|W}:W\to V$ is birational. The restriction of $g$ to $W$ is rational by assumption on $g$, so that the composition is again rational, giving $(i)$.
\end{proof}

In particular in the central case, we deduce from
Theorem \ref{bijfinitebiratbis-reg} an hereditary version of Proposition \ref{bijfinitebirat}.

\begin{prop}
\label{bijfinitebirat2}
Let $\pi:Y\rightarrow X$ be a finite birational map between
algebraic sets where $X$ is
central. Let us assume that $\pi$ is bijective and hereditarily birational. Then, one has the following properties:
\begin{enumerate}
 \item $Y$ is central. 
 \item The morphism $\SRR(X)\rightarrow \SRR(Y)$ $f\mapsto f\circ\pi$
   is an isomorphism.
 \item The extension $\Pol(X)\rightarrow \Pol(Y)$ is centrally subintegral.
 \item The rational map $\pi^{-1}$ is hereditarily rational.
\end{enumerate}
\end{prop}

\begin{proof}
By Proposition \ref{prop-her} and Theorem \ref{bijfinitebiratbis-reg},
we only have to prove that $\ReSp \Pol(Y)\to\ReSp \Pol(X)$ is
injective. Let $\p\in\ReSp \Pol(X)$ and $\q_1,\q_2\in\ReSp \Pol(Y)$
lying over $\p$. Since $\pi_{|\Z(q_i)}:\Z(\q_i)\to\Z(\p)$ is
birational, for $i\in\{1,2\}$, and $\pi$ is bijective, it follows that $\Z(q_1)=\Z(q_2)$ and the real Nullstellensatz \cite[Thm. 4.1.4]{BCR} implies that $\q_1=\q_2$.
\end{proof}

\begin{rem}
An algebraic set $X$ has
a totally real normalization if $(\pi')_{\C}^{-1}(X)=X'$ where $\pi':X'\to X$ is the normalization map. It is shown in
\cite{FMQ-futur} that finite birational maps onto an algebraic set with
totally real normalization have remarkable properties. Notably, from
\cite[Thm. 3.9]{FMQ-futur} we know that a finite birational bijection with
target an algebraic set with totally real normalization admits a rational inverse which is hereditarily rational on the central locus, and moreover it is an homeomorphism for the
Zariski topology.
\end{rem}


\section{Weak-normalization and seminormalization relative to the central locus}\label{WeakSemicentralLocus}

We introduce in this section the weak-normalization and seminormalization of an algebraic set, relative to its central locus. We develop the theory in parallel with the classical notion of normalization (and make some connection too with the biregular normalization defined in \cite{FMQ-futur}), meaning in particular that we define them via integral closures in well-chosen rings of functions. We make latter the connection with the algebraic notions of weak-normalization and seminormalization  relative to its central locus considered in section \ref{CentrallyNormalizationRing}.

\subsection{Integral closure in rings of continuous functions}

Let $X$ be an algebraic set. The normalization $X'$ of $X$ is the algebraic set whose
ring of polynomial functions is the integral closure of $\Pol(X)$ in
$\K(X)$. The biregular normalization $X^b$ of $X$ is the algebraic set whose
ring of polynomial functions is the integral closure of $\Pol(X)$ in $\SO(X)$.
Taking advantage of the sequence of inclusions
 $$\Pol(X)\hookrightarrow \SO(X)\hookrightarrow \SRR(\Cent X)\hookrightarrow
\SR(\Cent X)\hookrightarrow \K(X)$$
we define, using Proposition \ref{intermediatering}, two intermediate algebraic sets between a given algebraic set and its normalization. 
\begin{defn}
\label{defnormalizations}
Let $X$ be an algebraic
set. 
\begin{enumerate}
\item The weak-normalization relative to the central locus (or
  $w_c$-normalization) $X^{w_c}$ of $X$ is the algebraic set whose ring
  of polynomial functions is the integral closure of $\Pol(X)$ in
  $\SR(\Cent X)$.
\item The seminormalization relative to the central locus (or
  $s_c$-normalization) $X^{s_c}$ of $X$ is the algebraic set whose ring
  of polynomial functions is the integral closure of $\Pol(X)$ in
  $\SRR(\Cent X)$.
\end{enumerate}
The natural finite birational maps $\pi^{w_c}:X^{w_c}\rightarrow X$
and $\pi^{s_c}:X^{s_c}\rightarrow X$
are respectively called the $w_c$-normalization map and the
$s_c$-normalization map. The algebraic set $X$ is called centrally weakly-normal if $X=X^{w_c}$ 
and centrally seminormal if
$X=X^{s_c}$.
\end{defn}

In particular the sets $X^{w_c}$ and
$X^{s_c}$ are intermediate algebraic sets between $X^b$
and $X'$, and we derive a sequence of finite birational maps 
$$X'\to X^{w_c} \to X^{s_c} \to X^b \to X$$
that will be the object of study in the rest of the paper.

\begin{ex}\label{ex-Ko}
As a first example, consider Koll\'ar surface $X=\Z(x^3-y^3(1+z^2))$. 
The polynomial ring of the normalization is given by $\Pol(X')=\Pol(X)[x/y]$ and the rational function $x/y$ can be extended to a
  continuous function on $X$. As a consequence $X'$ coincides with $X^{w_c}$, whereas $x/y$ is not hereditarily rational. Note that $X^{s_c}\not=X^{w_c}$, more precisely $X^{s_c}$ has coordinate ring $\Pol(X)[x^2/y]$. 
  \end{ex}

\begin{rem}
\label{explication} It would make sense algebraically to replace $\Pol(X)$ by
$\SO(X)$ in Definition \ref{defnormalizations}. However in general, the
integral closure of $\SO(X)$ in $\K(X)$ or $\SR(\Cent X)$ or $\SRR(\Cent X)$ is
not the ring of regular functions of an intermediate algebraic set
between $X$ and $X'$ (see \cite{FMQ-futur}). This justify why we have decided to
work with $\Pol(X)$ rather than $\SO(X)$ (contrarily to \cite{MR}).
\end{rem}

From the very definitions of $w_c$-normalization and $s_c$-normalization, we get that, for an irreducible algebraic set $X$, if $\SR(\Cent X)=\SRR(\Cent X)$ then $X^{s_c}=X^{w_c}$. In particular, this is the case for curves or more generally varieties with isolated singularities.

On a complex algebraic variety which is normal, a rational function admitting a continuous extension at its poles is automatically regular by Riemann extension theorem. We characterize, by universal properties closed to the Riemann extension theorem, the algebraic varieties that are
``normal'' for the different normalizations we have encountered. 

\begin{prop}
	\label{unpropnormal}
	Let $X$ be a real algebraic set. Then
	\begin{enumerate}
		\item $X$ is normal if and only if every rational function in $\K(X)$
		bounded on
		$X_{\C}$ is polynomial on $X$.
		\item $X$ is seminormal if and only
		if every rational function in $\K(X)$ integral over $\Pol(X)$ and continuous on $X_{\C}$ is polynomial
		on $X$.
		\item $X$ is biregularly normal if and only if every regular function on $X$ integral over $\Pol(X)$ is
		polynomial on $X$.
		\item $X$ is centrally seminormal if and only if
		every function in $\SRR(\Cent X)$ integral over
                $\Pol(X)$ is the restriction to $\Cent X$ of a
                polynomial function on $X$.
		\item $X$ is centrally weakly-normal if and only
		if every function in $\SR(\Cent X)$ integral over
                $\Pol(X)$ is the restriction to $\Cent X$ of a
                polynomial function on $X$.
	\end{enumerate}
\end{prop}

\begin{proof}
	Notice that a regular function on $X_{\C}$ is polynomial on $X_{\C}$. By the Riemann extension theorem then we get (1). Statement (2) follows
	from \cite{AB, AN} and \cite{T}. The last three statements are direct
	consequences from the definitions of the biregular, centrally weakly
	and centrally semi-normalizations.
\end{proof}

\begin{rem} 
	We see by Proposition
	\ref{unpropnormal} that $w_c$-normalization
	doesn't correspond to $w_c$-normalization
 of the irreducible components like
	classical normalization does (and the same remark holds for the $s_c$-normalization). More precisely, it may happen that
	every irreducible component of a non centrally weakly-normal algebraic set is centrally weakly-normal. Consider for instance the union of three
	distinct two-by-two lines passing through the origin in $\R^2$. In that case the $w_c$-normalization is the union of the three
	axes in $\R^3$ and it is also the $s_c$-normalization (see the next
	section). 
\end{rem}

We can identify the functions in $\Pol(X^{w_c})$ with the rational functions in
$\K(X)$ that are integral over $\Pol(X)$ and that admit a continuous
extension to $\Cent X$. Similarly, the functions in $\Pol(X^{s_c})$ are the rational functions in
$\K(X)$ that are integral over $\Pol(X)$ and that admit an hereditarily
rational extension to $\Cent X$. We state this fact as a lemma for further reference.

\begin{lem}
\label{equivdefnormalizations}
Let $X$ be an algebraic set and $\pi':X'\rightarrow X$ be the normalization map.
\begin{enumerate}
\item The polynomial functions on $X^{w_c}$ form the subring of $\Pol(X')$ given by $$\Pol(X^{w_c})=\{g\in\Pol(X')|\,\exists f\in\SR(\Cent X)\,{\rm
    such\,  that}\,f\circ\pi'=g\,{\rm on}\,\Cent X'\}.$$
\item The polynomial functions on $X^{s_c}$ form the subring of $\Pol(X')$ given by  $$\Pol(X^{s_c})=\{g\in\Pol(X')|\,\exists f\in\SRR(\Cent X)\,{\rm
    such\,  that}\,f\circ\pi'=g\,{\rm on}\,\Cent X'\}.$$
\end{enumerate}
\end{lem}

We aim to prove that the ring $\Pol(X^{w_c})$ is the weak normalization relative to the central locus of the ring $\Pol(X)$, together with the analogous statement for $\Pol(X^{s_c})$ with respect to the seminormalization. Next result can be viewed as a first step in this direction.

\begin{prop}\label{cor-weak-bij} Let $X$ be an algebraic set.
\begin{enumerate}
\item The map $\pi^{w_c}:X^{w_c}\to X$ is a bijection between the central loci and its inverse map is rational continuous on the center of $X$. 
\item The map $\pi^{s_c}:X^{s_c}\to X$ is a bijection between the central loci and its inverse map is hereditarily rational on the center of $X$. 
\end{enumerate}
\end{prop}

\begin{proof} Consider the case of the weak-normalization relative to the central locus. Denote by $\psi$ the map $X'\to X^{w_c}$. Let $g\in \Pol(X^{w_c})$. By
Lemma \ref{equivdefnormalizations}, there exists a continuous function
   $f\in\SR(\Cent X)$ such
   that $g\circ\psi=f\circ\pi'$ on
   $\Cent X'$. By Lemma \ref{fondamental}, we get that
   $g=f\circ\pi^{w_c}$ on $\Cent X^{w_c}$. As a consequence the map $\pi^{w_c}:\Cent X^{w_c}\to \Cent X$ is
   bijective or equivalently $\pi^{w_c}:X^{w_c}\to X$ is a bijection on the central loci and its inverse is rational continuous by Theorem
   \ref{bijfinitebiratbis}. The proof in the seminormal case is similar using Theorem \ref{bijfinitebiratbis-reg}.
\end{proof}

Before establishing the universal property satisfied by the weak-normalization (respectively seminormalization) relative to the central locus, and make the connection with its algebraic counterpart, we discuss several examples of curves and surfaces in order to give the reader some familiarity with the new sets introduced. We provide in particular examples proving that the different notions of normalization give rise to different algebraic sets.

Let us begin the discussion with the case of curves. Recall that in that
situation rational continuous and hereditarily rational functions
coincide, so that the $w_c$-normalization and the $s_c$-normalization
will give the same sets. 
In the case of reducible curves, two crossing lines in the plane will give a centrally seminormal curve contrarily to three lines crossing in a common point.

\begin{ex} \label{ex-lines}
\begin{enumerate}
\item Two intersecting lines in the plane are centrally seminormal. Actually, an integral continuous rational function on it coincides, on each lines, with a polynomial function since a line is normal. And the data of a polynomial on each line, with the same value at the intersection point, defines a polynomial function on the union of the lines.
\item The picture is different with three lines in the plane intersecting in a common point. Here the data of a polynomial on each line, with the same value at the intersection point, does not necessarily define a polynomial function on the union of the lines. The central seminormalization will be given by three non coplanar lines in $\R^3$ 
cf. Proposition \ref{seminormcurve}.
\end{enumerate}
\end{ex}

We focus now on irreducible plane curves.
In order to make the comparison with the biregular normalization, we recall that the biregular normalization of an
algebraic set may be constructed by normalizing the non-real locus of
its complexification \cite[Prop. 4.6]{FMQ-futur}.

\begin{ex}
\begin{enumerate}
\item The $w_c$-normalization of the cuspidal plane cubic
  $X=\Z(y^2-x^3)$ coincides with its normalization, and $X=X^b\subset X^{w_c}=X'$. 
  \item The nodal plane curve $X=\Z(y^2-x^2(x+1))$ is centrally seminormal, and $X=X^b=X^{w_c}\subset X'$.
\item \label{ex-weak-curve} The seminormalization with respect to the central locus can be different from $X$ and $X'$. The origin is the unique singular point of $X=\Z(y^2-x^4(x+1))$,
  where two distinct branches intersect with tangency. 
  Note that the rational function $y/x$ satisfies $(y/x)^2=x^2(x+1)$ and in
  this case 
  $$\Pol(X')=\Pol(X)[y/x^2]=\Pol(X)[z]/(z^2-x-1,
    x^2z-y).$$
It follows that
 $X^{w_c}$ has coordinate ring
 $\Pol(X^{w_c})=\Pol(X)[y/x]$. Notice that $X=X^b$ here.


\end{enumerate}
\end{ex}

We discuss now some examples of surfaces in $\R^3$. We already mentioned that Koll\'ar surface is an example where the weak-normalization relative to the central locus 
differs from the seminormalization relative to the central
locus. Note however that in the case of isolated singularities
(which are not necessarily normal if the complex singularities are not
isolated), the rational continuous and hereditarily rational functions coincide, so that the $w_c$-normalization and the $s_c$-normalization will be the same.

Consider first reducible surfaces. 

\begin{ex}
\begin{enumerate}
\item The union of two intersecting planes in $\R^3$ defines a centrally seminormal surface, for similar reason that in Example \ref{ex-lines}. We can also use the stability 
of the central seminormalization under product that will be proved in
Proposition \ref{cor-prod-semi} to conclude, since a line is centrally seminormal, and two intersecting lines too.
\item Similarly the union of the coordinates hyperplanes in $\R^n$ is centrally seminormal. 
\end{enumerate}
\end{ex}

Finally we discuss two classical examples of umbrellas.

\begin{ex}
\begin{enumerate}
\item The normalization of the (non-central) Whitney umbrella $X=\Z(x^2-y^2z)$ is the affine plane via $\Pol(X')\simeq\R[y,x/y]$. Note that the normalization map is not a bijection on the central loci, so that $X=X^b=X^{w_c}=X^{s_c}$. 
\item The normalization of the (non-central) Cartan umbrella $X=\Z(x^3-(x^2+y^2)z)$ is given by $\Pol(X')=\Pol[X][yz/x]$. The rational function $yz/x$
admits a continuous extension by zero along the $y$-axis, which is the
intersection of its indeterminacy locus with the central part of $X$,
so that $X'=X^{w_c}$. Moreover the restriction of $yz/x$ to the
$y$-axis is a constant function so it is still rational, so that
$X^{w_c}=X^{s_c}$. Since $yz/x$ is not regular on the Cartan umbrella, it
can be concluded that $X=X^b$.
\end{enumerate}
\end{ex}

\subsection{Universal properties for the $w_c$-normalization }

We give the universal property satisfied by the $w_c$-normalization in light of the universal property of the normalization. 

\begin{thm}
\label{propunivcentweak}
Let $X$ be an algebraic
set and let $X'$ be its normalization. Let $Y$ be an algebraic set
equipped with a finite birational map $\pi:Y\to X$. 

Then $\pi$ induces
a bijection from $\Cent Y$ to $\Cent X$ if and only if
$\pi^{w_c}:X^{w_c}\to X$ factorizes through $\pi$.
\end{thm}

\begin{proof}
Notice first that from Proposition \ref{cor-weak-bij} we know that the restriction of $\pi^{w_c}:X^{w_c}\to X$ to the central loci is a bijection. 
Let $Y$ be an algebraic set with finite birational maps $\pi:Y\to X$ and $\varphi:X'\to
Y$.

Assume first that $\pi^{w_c}$ factorizes through $\pi$. By Theorem
\ref{bijfinitebiratbis} the extension $\Pol(X)\to \Pol(X^{w_c})$ is
$w_c$-subintegral and thus $\Pol(X)\to \Pol(Y)$ is also
$w_c$-subintegral (Proposition \ref{intermedwc}). From Theorem
\ref{bijfinitebiratbis} again then it follows that $\pi:\Cent Y\rightarrow\Cent X$ is a bijection.

Assume now that $\pi:\Cent Y\rightarrow\Cent X$ is a bijection. Let
$g\in\Pol(Y)$. By Theorem \ref{bijfinitebiratbis} there exists a continuous function
   $f\in\SR(X)$ such
   that $g=f\circ\pi$ on $\Cent Y$. By Lemma
   \ref{fondamental}, then $g\circ\varphi=f\circ\pi'$ on
   $\Cent X'$ and thus $g\circ\varphi\in\Pol(X^{w_c})$ (Lemma \ref{equivdefnormalizations}). Since the
   composition by $\varphi$ gives the inclusion
   $\Pol(Y)\subset\Pol(X')$
then we get an inclusion $\Pol(Y)\subset\Pol(X^{w_c})$ and thus $\pi^{w_c}:X^{w_c}\rightarrow X$ uniquely factors through
$\pi:Y\rightarrow X$.
\end{proof}

From Theorems \ref{bijfinitebiratbis} and \ref{propunivcentweak}, we deduce that
  the $w_c$-normalization $X^{w_c}$ of $X$ is the
biggest set among the intermediate algebraic sets between $X$ and $X'$
whose central loci are birationally continuously isomorphic with the
central locus of $X$.

Note that if the $w_c$-normalization map $\pi^{w_c}:X^{w_c}\to
X$ is bijective by restriction to the central locus, it is not a bijection in general even if we assume $X$ to be central. Indeed, consider the surface $X=\Z((y^2+z^2)^2-x(x^2+y^2+z^2))$ from Example \ref{grospoint2}.
Since $\pi':X'\to X$ is a bijection by restriction to the central locus, it follows from Theorem
\ref{propunivcentweak} that $X^{w_c}=X'$. 

\begin{rem} 
The $w_c$-normalization of a (non-normal) toric variety coincides with its normalization. Indeed, the normalization is obtained by saturation of the semi-groups, giving rise to a normal toric variety with the same torus decomposition into orbits. In particular a toric variety and its normalization are in bijection. 
\end{rem}

We prove that in the geometric setting we recover the algebraic $w_c$-normalization.
\begin{thm}
\label{equivrealweaknorm}
Let $X$ be an algebraic
set and let $X'$ be its normalization. We have $$\Pol(X)^{w_c}=\Pol(X^{w_c}).$$
\end{thm}

\begin{proof}
The morphism $\Pol(X)\rightarrow \Pol(X)^{w_c}$ being $w_c$-subintegral (see Proposition \ref{UniversalWCrings}),
it induces a bijection in restriction to the central maximal ideals
and thus a bijection from $\Cent Y$ onto $\Cent X$ where $Y$ is the
algebraic set with $\Pol(X)^{w_c}$ as ring of polynomial functions
(see
Proposition \ref{intermediatering}). By the universal property Theorem \ref{propunivcentweak}, on gets that
$Y$ is an intermediate algebraic set between $X$ and $X^{w_c}$, and hence 
$\Pol(X)^{w_c}=\Pol(Y)\subset\Pol(X^{w_c})$
(recall that $\Pol(X)^{w_c}$ and $\Pol(X^{w_c})$ are both viewed in 
$\Pol(X')$).

To show the converse inclusion, let us note that it follows from
Theorem \ref{bijfinitebiratbis} and Proposition \ref{cor-weak-bij} that the extension ring $\Pol(X)\rightarrow\Pol(X^{w_c})$
is $w_c$-subintegral. Then, by the universal property Proposition \ref{UniversalWCrings}, one derives a factorization morphism
$\Pol(X^{w_c})\rightarrow\Pol(X)^{w_c}$.
\end{proof}


As direct consequence of Theorem \ref{equivrealweaknorm} and Proposition \ref{WeakIdempotency}, 
one gets the idempotency of the $w_c$-normalization :
\begin{cor}
\label{idemp1}
Let $X$ be an algebraic set. Then, $(X^{w_c})^{w_c}=X^{w_c}$.
\end{cor}


The inverse of a finite birational map that induces a bijection between the central loci is known 
to be rational continuous
(see Theorem \ref{bijfinitebiratbis}). Next result gives a stronger statement, when we assume
that the target variety is centrally weakly normal. It is a
direct application of our universal properties (Theorem \ref{propunivcentweak} and Proposition \ref{UniversalWCrings}) and it gives a real version of
\cite[Cor. 2.8]{LV}.
\begin{prop}
\label{lvreel}
Let $X$ be an algebraic
set. Suppose that $X$ is centrally weakly-normal and that
$\varphi:Y\rightarrow X$ is a finite birational polynomial map with
$Y$ an algebraic set. Then $\varphi$ is a bijection on the central loci
if and only if $\varphi$ is an isomorphism.
\end{prop}

Note that it is not possible to remove the ``finite'' hypothesis in the
  previous proposition. For instance, let $X$ be the
  nodal curve given by $y^2=x^2(x+1)$ in $\R^2$, and $Y$ be the
  hyperbola given by $xy=1$ in $\R^2$. They are both centrally weakly normal and
  both in polynomial bijection with the punctured line $\R\setminus \{1\}$
  however they are not isomorphic curves since $X$ is
singular whereas $Y$ is smooth.

\vskip 2mm

Centrally weakly normal sets are stable under the product of varieties. 
\begin{prop}
	\label{cor-prod-weak} Let $X$ and $Y$ be centrally weakly-normal algebraic sets. Then $X\times Y$ is centrally weakly-normal.
\end{prop}

\begin{proof} We use the same strategy as in \cite[Cor. 2.13]{LV}.
	Let $f$ be a rational continuous function on $\Cent (X\times Y)$ which is
	integral over $\Pol(X\times Y)$. By Proposition
	\ref{unpropnormal}, we have to show that $f$ is polynomial on $\Cent (X\times Y)$.
	Then, for any $x\in \Cent X$, the
	restriction $f_x$ of $f$ to $\{x\}\times Y$ satisfies an integral
	equation over $\Pol(Y)$. 
	Note however that, if $f_x$ is not necessarily a rational function on
	$Y$,  there exists a Zariski dense subset $U$ in $X$ such that $f_x$
	is rational for any $x\in U$. By intersecting $U$ with the union of
	the non-singular loci of the irreducible components of $X$ then we may
	assume $U\subset \Cent X$. By $w_c$-normality of $Y$, it
	follows that $f_x$ is polynomial on $\Cent Y$ for any $x \in U$. 
	Similarly, there exists a Zariski dense subset $V$ in $Y$ such that
	$V\subset \Cent Y$ and
	$f_y$ is polynomial on $\Cent X$ for any $y \in V$.
	
	We want to conclude that $f$ coincides with a polynomial function on $\Cent 
	(X\times
	Y)$. We know by Palais \cite{Pa} that $f$ is polynomial on $U\times V$,
	so that there exists a polynomial function $p\in \Pol(X\times Y)$ such
	that $f=p$ on $U\times V$. 
	Since $f$ is continuous, it implies that $f=p$ on $\Cent (X\times Y)$
	i.e $f$ is polynomial on $\Cent (X\times Y)$. As a consequence $X \times Y$ is centrally weakly-normal.
\end{proof}

\subsection{Universal property for the $s_c$-seminormalization}

Here is the universal property for the seminormalization relative to the central locus.

\begin{thm}
\label{propunivseminormalization}
Let $X$ be an algebraic
set and let $X'$ be its normalization. Let $Y$ be any algebraic set equipped with 
a finite birational map $\pi:Y\rightarrow X$.

Then $\pi^{-1}$ is hereditarily rational on the center of $X$ if
and only if $\pi^{s_c}:X^{s_c}\to X$ factorizes through $\pi$.
\end{thm}

\begin{proof}
Notice first that from Proposition \ref{cor-weak-bij} we know that the rational inverse of $\pi^{s_c}$ is hereditarily rational on the center of $X$.

Assume $\pi^{s_c}:X^{s_c}\to X$ factorizes through $\pi$. We know by Proposition \ref{cor-weak-bij} and Theorem \ref{bijfinitebiratbis-reg} that $\Pol(X)\to\Pol(X^{s_c})$ is
centrally subintegral. Then it follows from Proposition \ref{intermedwc}
that $\Pol(X)\to\Pol(Y)$ is
centrally subintegral. From Theorem
\ref{bijfinitebiratbis-reg} we conclude that the rational inverse of $\pi:Y\to X$ 
is hereditarily rational on the center of $X$.

Assume now $\pi:Y\rightarrow X$ is a
finite birational map whose inverse is hereditarily rational on the center of $X$. By the universal property of the normalization,
the normalization map $\pi'$ factorizes by $\pi$ and let
$\psi:X'\to Y$ be the finite birational map such that
$\pi'=\pi\circ\psi$. Let $g\in\Pol(Y)$. By Theorem \ref{bijfinitebiratbis-reg}, there
exists $f\in\SRR(\Cent X)$ such that $g=f\circ\pi$ on
$\Cent Y$. By Lemma \ref{fondamental}, we get
$g\circ\psi=f\circ\pi'$ on $\Cent X'$ and thus
$g\circ\psi\in\Pol(X^{s_c})$ (Lemma \ref{equivdefnormalizations}). It shows that $\Pol(Y)\subset \Pol(X^{s_c})$
and the proof is done.
\end{proof}

From Theorems \ref{bijfinitebiratbis-reg} and \ref{propunivseminormalization}, we deduce that
  the $s_c$-normalization $X^{s_c}$ of $X$ is the
biggest set among the intermediate algebraic sets between $X$ and $X'$ equipped with a map to $X$ whose rational inverse is hereditarily rational on the center of $X$.

Let us now show how we recover the algebraic 
$s_c$-seminormalization:
\begin{thm}
\label{equivrealseminorm}
Let $X$ be an algebraic set. We have $\Pol(X)^{s_c}=\Pol(X^{s_c}).$
\end{thm}

\begin{proof}
Let $\pi':X'\rightarrow X$ be the normalization map.
Let us see first that we have $\Pol(X)^{s_c}\subset\Pol(X^{s_c})$. Let $Y$ be the 
algebraic set associated to the ring $\Pol(X)^{s_c}$ (see
Proposition \ref{intermediatering}). 
The morphism $\Pol(X)\rightarrow \Pol(X)^{s_c}$ being
centrally subintegral (see Proposition \ref{UniversalNCrings}), from Theorem \ref{bijfinitebiratbis-reg} it follows
that the finite birational map $Y\to X$ induces a bijection from $\Cent Y$ onto $\Cent X$ and its rational inverse is hereditarily rational on the center of $X$. By the universal property of Theorem \ref{propunivseminormalization}, one gets that
$Y$ is an intermediate algebraic set between $X$ and $X^{s_c}$, and hence 
$\Pol(X)^{s_c}=\Pol(Y)\subset\Pol(X^{s_c})$.

To show the converse inclusion, it follows from Proposition \ref{cor-weak-bij} and Theorem \ref{bijfinitebiratbis-reg} that the extension ring $\Pol(X)\rightarrow\Pol(X^{s_c})$
is centrally subintegral. Then, by the universal property Proposition \ref{UniversalNCrings}, one derives a factorization morphism
$\Pol(X^{s_c})\rightarrow\Pol(X)^{s_c}$. 
\end{proof}

Combined with the idempotency properties of the algebraic $s_c$-seminormalization (Proposition \ref{SemiIdempotency}), we obtain :
\begin{cor}
\label{idemp2}
Let $X$ be an algebraic set. Then 
\begin{enumerate}
 \item $(X^{s_c})^{s_c}=X^{s_c}$,
 \item $(X^{w_c})^{s_c}=X^{w_c}$,
 \item $(X^{s_c})^{w_c}=X^{w_c}$.
\end{enumerate}
\end{cor}

An immediate application of Theorem \ref{propunivseminormalization}
gives a seminormal version of Proposition \ref{lvreel}.
\begin{prop}
\label{lvreel2}
Let $X$ be an algebraic
set. Suppose that $X$ is centrally seminormal and that
$\varphi:Y\rightarrow X$ is a finite birational polynomial map with
$Y$ an algebraic set. Then $\varphi^{-1}$ is hereditarily rational on the center of $X$ if and only if $\varphi$ is an isomorphism.
\end{prop}




Centrally seminormal sets are stable under the product of varieties. 
\begin{prop}\label{cor-prod-semi} Let $X$ and $Y$ be centrally seminormal algebraic sets. Then $X\times Y$ is centrally seminormal.
\end{prop}

\begin{proof} We use the same proof as in Proposition	\ref{cor-prod-weak}. Note that the proof is even simpler since the
	restriction of an hereditarily rational function is rational so that
	we can choose $U$ and $V$ to be respectively the union of the
	non-singular loci of the irreducible components of $X$ and $Y$.
\end{proof}

The following proposition provides a lot of examples of varieties for
which $w_c$-normalization and $s_c$-normalization coincide but the
rings of continuous rational functions and hereditarily rational
function may be different.

\begin{prop}
  \label{totreal}
Let $X$ be a central algebraic set with totally real
normalization. If $X^{w_c}$ is central then $X^{w_c}=X^{s_c}$.
\end{prop}

\begin{proof}
By Proposition \ref{bijfinitebirat2} then $X^{w_c}$ is central. If $X$ has totally real
normalization then it follows from \cite[Thm. 3.9]{FMQ-futur} that the
bijective finite birational map $\pi^{w_c}:X^{w_c}\to X$ admits an inverse which is 
hereditarily rational and the proof is done.
\end{proof}



\subsection{Up to biregular isomorphisms}
From \ref{sectiongeoalg}, the constructions of the weak and semi normalizations we have made for
real algebraic sets can be extended to affine real algebraic varieties
(such that the set of real closed points is Zariski dense in the set
of closed points). It is in particular clear that isomorphic affine real
algebraic varieties have isomorphic weak and semi normalizations. We
prove moreover in this section that biregular affine real algebraic varieties
have biregular weak and semi normalizations. Combined with the fact
that the real closed points of a quasi-projective variety defined over
$\R$ are always included in an affine variety via a regular map, this
proved that the weak and semi normalizations are well defined for
quasi-projective real algebraic varieties in the sense of \cite{BCR}.

To go toward the proof, we begin with the fact that our constructions
commute with the localization along the multiplicative part $S$
defining the regular functions on an algebraic set $X\subset \R^n$, namely $S=1+\sum \Pol(X)^2$. This result has to be compared with the results in \ref{sect-comm}.

\begin{lem}\label{oneplussquaresseminorm} Let $X$ be an algebraic set and $S=1+\sum \Pol(X)^2$. Then $S^{-1}\Pol(X^{w_c})\simeq \SO(X)^{w_c}$ and $S^{-1}\Pol(X^{s_c})\simeq \SO(X)^{s_c}$.
\end{lem}

\begin{proof}
Let us deal with $w_c$-normalization. We aim to prove that the second inclusion in the sequence of natural inclusions
$$\SO(X) \to S^{-1}\Pol(X^{w_c}) \to \SO(X)^{w_c}$$
is an isomorphism. Let $f$ be an element in $\SO(X)^{w_c}$. Note that
$\SO(X)^{w_c}$ is included in $\SO(X)'$, this latter being isomorphic
to $S^{-1}\Pol(X')$ since normalization commutes with localization, so
that there exists $s\in S$ such that $sf$ belongs to $\Pol(X')$. So it
is sufficient to prove that $sf$ belongs to $\Pol(X^{w_c})=\Pol(X)^{w_c}$. Note that we
have a natural bijection $\RCent \Pol(X)\to \RCent \SO(X)$, $\p\mapsto
S^{-1}\p$ that gives also a bijection by restriction on maximal ideals on
both sides (see \cite{FMQ-futur}).
Let $\mathfrak m$ be a central maximal ideal in $\Pol(X)$, and denotes by $\tilde {\mathfrak m}$ the corresponding maximal ideal in $\SO(X)$. Note that $\Pol(X)_{\mathfrak m}\simeq \SO(X)_{\tilde {\mathfrak m}}$ and $\Pol(X')_{\mathfrak m}\simeq \SO(X')_{\tilde {\mathfrak m}}$. Then by assumption
$$(sf)_{\tilde {\mathfrak m}}\in \SO(X)_{\tilde {\mathfrak m}}+ \JRadCe \SO(X')_{\tilde {\mathfrak m}}$$
so that 
$$(sf)_{\mathfrak m} \in \Pol(X)_{\mathfrak m}+ \JRadCe \Pol(X')_{\mathfrak m},$$
proving that $sf \in \Pol(X^{w_c})$ as expected.

For $s_c$-normalization, we proceed similarly considering central prime ideal moreover.
\end{proof}

With this in hand one can prove the desired result.

\begin{thm}\label{bireg} Let $X$ and $Y$ be biregular real algebraic sets. Then $X^{w_c}$ is biregular to $Y^{w_c}$, and $X^{s_c}$ is biregular to $Y^{s_c}$.
\end{thm}

\begin{proof}
Let us prove the result about $w_c$-normalization, we proceed likewise for 
$s_c$-normalization.

Using Theorem \ref{equivrealweaknorm} and Lemma
\ref{oneplussquaresseminorm}, one has a canonical commutative diagram
of injective morphisms
$$\begin{array}{ccccccc}
	\Pol(X)&\rightarrow &\SO(X)&\simeq& \SO(Y)&\leftarrow &\Pol(Y)\\
	\downarrow&&\downarrow&&\downarrow&&\downarrow\\
	\Pol(X)^{w_c}&\rightarrow &S^{-1}\Pol(X)^{w_c}\simeq \SO(X)^{w_c}&\simeq &
	\SO(Y)^{w_c}\simeq T^{-1}\Pol(Y)^{w_c}&
	\leftarrow&\Pol(Y)^{w_c}\\
	&\searrow&\downarrow&&\downarrow&\swarrow\\
	&&S_{w_c}^{-1}\Pol(X^{w_c})=\SO(X^{w_c})& &T_{w_c}^{-1}\Pol(Y^{w_c})=\SO(Y^{w_c})\\
\end{array}$$
where we have set $S=1+\sum \Pol(X)^2$, 
$S_{\omega_c}=1+\sum \Pol(X^{\omega_c})^2$,
$T=1+\sum \Pol(Y)^2$, $T_{\omega_c}=1+\sum \Pol(Y^{\omega_c})^2$.

Producing a canonical morphism from one bottom ring to the other one will show, by symmetry, that they both are canonically isomorphic.

Using the horizontal middle morphisms, one has a canonical morphism 
$$\Pol(X)^{w_c}\rightarrow T_{w_c}^{-1}\Pol(Y^{w_c}).$$

To show that it factorizes through 
$S_{w_c}^{-1}\Pol(X^{w_c})$, it is enough to prove that any element of 
$S_{w_c}$ becomes invertible in 
$T_{w_c}^{-1}\Pol(Y^{w_c})$.

Let us then consider $s=1+\sum a_i^2$ where
$a_i\in \Pol(X^{w_c})$.
There exists $b_i\in \Pol(Y^{w_c})$ and 
$t_i\in T$ such that $a_i$ is sent to 
$b_i/t_i$ in $T^{-1}\Pol(Y)^{w_c}$.
Since $t_i^2\in T\subset T_{w_c}$ is invertible
in $T_{w_c}^{-1}\Pol(Y^{w_c})$, it 
is enough to show that $f=t^2+\sum c_i^2$ 
is invertible in $T_{w_c}^{-1}\Pol(Y^{w_c})$ for $t\in T$ and $c_i\in\Pol(Y)^{w_c}$. 
Since $t=1+\sum u_i^2$ with $u_i\in 
\Pol(Y)$, we get that $f\in T_{w_c}$.
This gives the result.
\end{proof}

Note that the same proof adapts to the case of the normalization
(stated in \cite[page 75]{BCR}) since the normalization process commutes with localization.
\subsection{Commutation with localization}\label{sect-comm}
It is known, for example from \cite[Corollary 1.6]{LV}, that the standard seminormalization $^+A$ commutes with localization.
The proof of this result follows quite directly from the conductor criterion for seminormal extensions as stated in \cite[Proposition 1.4]{LV}. 
Unfortunately, this criterion appears not to have a clear counterpart in our real central setting. Hence, we produce here a new argument
based on an induction on the dimension.
Namely, we get that the central seminormalization commutes with
localization at a prime ideal: 
\begin{thm}\label{thm-loca}
	Let $X$ be an algebraic set and $\p$ be a prime ideal in $\Pol(X)$. Then
	$$(\Pol(X)_{\p})^{s_c}=\Pol(X^{s_c})_{\p}.$$
\end{thm}

Remind that the commutation of the integral closure with the standard normalization is formal. It is no more formal for $s_c$-normalization although it is a still an integral closure by Theorem \ref{equivrealseminorm}. One reason is that the localized ring is no more the polynomial ring of a real algebraic variety.
Another reason is that the commutation is false for central $w_c$-normalization although this latter ring is still an integral closure by Theorem \ref{equivrealweaknorm}.

\begin{rem} The result is false for the $w_c$-normalization. Consider the normalization $X'=\mathcal Z(y^2-z)$ of the Whitney umbrella $X=\mathcal Z(y^2-zx^2)$ which is given by $\pi :X' \to X$, $(x,y,z) \mapsto (x,xy,z)$. Then $X^{w_c}=X$, and the natural inclusion
	$$\Pol(X)_{\m}=\Pol(X^{w_c})_{\m} \subset (\Pol(X)_{\m})^{w_c}$$
	is strict for $\m=(x,y,z)$ the maximal ideal of the origin in $X$. Actually, consider $f=y/x \in \K(X)$. Then $f\circ \pi=y$ vanishes at the unique preimage by $\pi$ of the origin in $X'$, therefore $f$ belongs to 
	$$\JRadCe \Pol(X')_{\m}\subset (\Pol(X)_{\m})^{w_c}.$$
	However if $qf \in \Pol(X)$ with $q\in \Pol(X)$, then $q$ must vanish along the polar locus of $f$ so that $q\in \m$. Therefore $f$ does not belong to $\Pol(X^{w_c})_{\m}$.
\end{rem}

We cut the proof of Theorem \ref{thm-loca} into several lemmas, starting with the first key Lemma \ref{lemm1} and also the following one which translates geometrically the belonging of a rational function to the $s_c$-normalization of the localization of $\Pol(X)$ at some prime ideal~: 
 
\begin{lem}\label{lemm2} Let $\pi : X' \to X$ denote the normalization map, and $f\in \K(X)$  be a rational function integral over $\Pol(X)$. Let $V$ be an irreducible component of the polar locus of $f$, with $V$ central in $X$. Assume that the localization $(f\circ \pi)_{\mathcal I(V)} $ along $\mathcal I(V)$ of the polynomial function $f\circ \pi$ on $X'$ satisfies
	$$(f\circ \pi)_{\mathcal I(V)}\in \Pol(X)_{\mathcal I(V)}+ \JRadCe{\Pol(X')_{\mathcal I(V)}}.$$
	Then there exists an algebraic subset $Z\subset V$ with $\dim Z<\dim V$ such that
	\begin{enumerate}
		\item the rational function $f$ can be extended continuously along $\Cent X \cap (V\setminus Z)$,
		\item the extension remains rational restricted to any irreducible subset of $V$ central in $X$ and not included in $Z$.
	\end{enumerate}
\end{lem}

\begin{proof} Let $W_1,\ldots,W_r\subset X'$ be the irreducible algebraic sets  central in $X'$ lying over $V$.

	By assumption on $f$, there exist $q\in \Pol(X) \setminus
        \mathcal I(V)$, $p\in \Pol(X)$ and $h\in \Pol(X')$ (we see $h$
        as a rational function on $X$ (integral over $\Pol(X)$)) such that 
	$$h\circ \pi \in \cap_{i=1}^r \mathcal I(W_i)$$
	and
	$$q f =p+h$$
	as rational functions on $X$.

	Let $\tilde Z$ be the algebraic subset of $V$ given by Lemma
	\ref{lemm1}. Let $x\in (V\setminus \tilde Z) \cap \Cent X$. By Lemma
	\ref{lemm1}, the set $\pi^{-1}(x) \cap \Cent X'$ is included in
	$W_1\cup \cdots \cup W_r$, so that $h\circ \pi$ vanishes on
	$\pi^{-1}(x) \cap \Cent X'$. This means that $h$ can be extended continuously at $x$ by setting $h(x)=0$. Note moreover that property $(2)$ holds for $h$ with respect to $\tilde Z$.
	
	Now let $Z$ be equal to $\tilde Z\cup (\mathcal Z(q) \cap V)$. We still have $\dim Z < \dim V$ since $q\notin \mathcal I(V)$, and $p$ being polynomial on $X$, properties $(1)$ and $(2)$ hold for $f=\frac{p+h}{q}$.
\end{proof}

Next lemma is crucial for the induction step in the proof of Theorem \ref{thm-loca}.

\begin{lem}\label{lemm3} Let $\pi : X' \to X$ denote the normalization map, and let $\p$ be a prime ideal in $\Pol(X)$. Let $f\in \K(X)$  be a rational function integral over $\Pol(X)$ such that for any $\q\in \RCent \Pol(X)$ included in $\p$, the localization $(f\circ \pi)_{\q} $ along $\q$ of the polynomial function $f\circ \pi$ on $X'$ satisfies
	$$(f\circ \pi)_{\q}\in \Pol(X)_{\q}+ \JRadCe{\Pol(X')_{\q}}.$$
	Then there exist $q\in \Pol(X) \setminus \p$ such that $qf\in \SRR(\Cent X)$.
\end{lem}

\begin{proof} Consider the central part $\pol^{C}(f)$ of the polar locus of $f$, namely
	$$\pol^{C}(f)=\overline{\Cent X \cap \pol(f)}^{Z}.$$
	Denote by $V_1,\ldots,V_s$ the irreducible components of $\pol^{C}(f)$
	which are central in $X$, and set $\q_i=\mathcal I(V_i)$. Assume $\q_i \subset \p$ if and only if $i\leq r'$, with $0\leq r'\leq r$. For any $i>r'$, choose $q_i\in \q_i \setminus \p$ and consider $f_1=f \prod_{i>r'} q_i$.

	For $i > r'$, the function $f_1$ can be extended continuously along $V_i$ by imposing the value zero (the rational function $f$ is locally bounded on $X$).

	For $i\leq r'$, take $Z_i\subset V_i$ given by Lemma \ref{lemm2} applied to $f_1$, and set $Z^1=\cup_{i\leq r'}Z_i$. Remark that $\dim Z^1 < \dim \pol^C(f)$.

	As a consequence :
	\begin{enumerate}
		\item the rational function $f_1$ can be extended continuously along $\Cent X \cap (\pol^C(f)\setminus Z^1)$,
		\item the extension remains rational restricted to any irreducible subset of $\pol^C(f)$ not included in $Z^1$, and central in $X$.
	\end{enumerate}

	We apply now the same construction to $f_1$ in place of $f$, and to $Z^1$ in place of $\pol^C (f)$. We obtain likewise a polynomial function $q_2\notin \p$ vanishing of the irreducible components of $Z^1$ central in $X$ and not containing $\mathcal Z(\p)$, and an algebraic subset $Z^2\subset Z^1$ with $\dim Z^2 < \dim Z^1$ such that :
	\begin{enumerate}
		\item the rational function $f_2=f_1q_2$ can be extended continuously along $\Cent X \cap (Z^1\setminus Z^2)$,
		\item the extension remains rational restricted to any irreducible subset of $Z^1$ not included in $Z^2$, and central in $X$.
	\end{enumerate}
	It remains to repeat finitely many times the construction to achieve the proof. 
\end{proof}

\begin{proof}[Proof of Theorem \ref{thm-loca}] Let $\pi :X'\to X$ be the normalization of $X$. Note that $(\Pol(X)_{\q})'=\Pol(X')_{\q}$ for any prime ideal $\q$ in $\Pol(X)$ since localization and normalization commutes.

	We aim to prove that the natural inclusion $\Pol(X^{s_c})_{\p} \subset (\Pol(X)_{\p})^{s_c}$ is an equality. Let $\phi$ belongs to the right hand side, so that by definition 
	$$\forall \q \in \RCent \Pol(X)_{\p},~~ \phi_{\q} \in (\Pol(X)_{\p})_{\q}+\JRadCe (\Pol(X')_{\p})_{\q}=\Pol(X)_{\q}+\JRadCe \Pol(X')_{\q}.$$
	In other words, there exists $a \in \Pol(X)\setminus \p$ such that $f=a\phi \in \K(X)$ is an integral function verifying : 
	$$\forall \q \in \RCent \Pol(X),~\q \subset \p, ~~ (f\circ \pi)_{\q}\in \Pol(X)_{\q} + \JRadCe \Pol(X')_{\q}.$$
	By Lemma \ref{lemm3}, there exists $q \in \Pol(X)\setminus \p$ such that $qf\in \SRR(\Cent X)$, meaning that $qf$ is in $\Pol(X^{s_c})$ by Theorem \ref{equivrealseminorm}. 
	As a consequence $f$ belongs to $\Pol(X^{s_c})_{\p}$ as required.
\end{proof}

\section{Curves}
The last section is devoted to centrally seminormal curves. In that situation, several particular phenomena appear. 
First of all, as already mentioned, continuous rational functions and hereditarily rational functions coincides on real curves, so the $w_c$-normalization and $s_c$-normalization coincide. We will show moreover that centrally seminormal curves are central. Nonetheless, the irreducible components of a centrally seminormal curve are also centrally seminormal. We end the section with a characterization of the singularities of centrally seminormal curves.

\vskip 2mm

We are going to prove that a centrally seminormal curve is central. In particular, the weak-normalization of any curve, relative to its central locus, becomes a central curve.

\begin{prop}\label{curve-cent} A centrally seminormal curve is central.
\end{prop}

\begin{proof}
Let $X\subset \R^n$ be a centrally seminormal curve. Assume $X$ is not central, with $x\in X$ a non-central point. Let
$Y\to X$ be the finite birational map obtained by normalizing the point
$x$. The fiber of $Y\to X$ over $x$ is empty and $\Cent Y$ is in
bijection with $\Cent X$. This is in contradiction with the universal property of the $w_c$-normalization (Theorem \ref{propunivcentweak}).
\end{proof}

It may happen that an irreducible component of a seminormal complex
algebraic set is not seminormal \cite[Ex. 2.11]{GT} but never for a
curve \cite[Cor. 2.9]{GT}. We prove a similar result for centrally
seminormal real curves.
\begin{prop}
\label{irredcurve}
  Let $X$ be a real algebraic curve. If $X$ is centrally
  seminormal then every irreducible component of $X$ is centrally
  seminormal.
\end{prop}

Notice that the converse of the statement of Proposition
\ref{irredcurve} is false as illustrated by the example of three lines in the plane intersecting in a common point.

\begin{proof}[Proof of Proposition \ref{irredcurve}]
Let $X_1,\ldots,X_t$ be the irreducible components of $X$, with $t\geq 2$. Note that $X$ is central by Proposition \ref{curve-cent} and that $X$ is the union of the central locus of its irreducible components, which intersect two-by-two in at most a finite of points.

Let $f\in\Pol(X_1^{w_c})$, we see the function $f$ both as a
polynomial function on $X_1^{w_c}$ or $X_1'$ and as a continuous
rational function on $\Cent X_1$ which is integral over
$\Pol(X_1)$. We complete $f$ with polynomial functions $p_k\in \Pol(X_k)$ such that 
the rational continuous functions of the $t$-tuple
$(f,p_2,\ldots,p_t)\in\SR(\Cent X_1)\times\cdots\times\SR(\Cent X_t)$ coincide on
the finite (if not empty) sets $\Cent X_i\cap\Cent X_j$ for all
$i,j\in\{1,\ldots,t\}$.
From Proposition \ref{ratcontreducible}, we see that
$(f,p_2,\ldots,p_t)$ induces a continuous rational function
$g\in\SR(\Cent X)$, which is integral over $\Pol(X)$ because $g\in
\Pol(X_1')\times\cdots\times\Pol(X_t')=\Pol(X')$ (Proposition
\ref{ReducedIntegralClosure}). Since $X$ is centrally seminormal then $g$ is polynomial on
$X$ as a consequence of Proposition \ref{unpropnormal}. By continuity, it
follows that $f$
is polynomial on $\Cent X_1$ and we conclude that $X_1$ is centrally
seminormal again by Proposition \ref{unpropnormal}.
\end{proof}

To finish to characterize centrally seminormal curves, we are lead to determine their singular points. In the case of seminormal complex curves, the singularities are only ordinary $k$-fold points. 
We say that a point $x$ of a complex algebraic curve $X$ 
is an ordinary $k$-fold point if $x$ is a point of multiplicity $k$
with $k$ linearly independent tangents. It other words, the singularity at $x$ is analytically
isomorphic to the union of the $k$ coordinate axes in $\C^k$ (see
\cite{Ko} for instance). Then by \cite{Bo}, a complex algebraic curve $X$ is seminormal if and only if the singularities of $X$ are ordinary $k$-fold points.

Coming back to a centrally seminormal real algebraic curve $X$, we know that $\Pol(X)$ is a
seminormal ring (in Traverso's sense). As a consequence $\Pol(X_{\C})$ and $X_{\C}$ are seminormal by \cite[Cor. 5.7]{GT} and therefore the singularities of the complexification $X_{\C}$ are ordinary $k$-fold points.

We will see below that when these complex singularities are moreover real, and with an additional condition it gives a characterization of centrally seminormal real curves. 

\begin{prop}
\label{seminormcurve}
Let $X$ be a real algebraic curve and $\pi':X'\rightarrow X$ be the normalization map.
Then
$X$ is centrally seminormal if and only if the following
properties are satisfied:
\begin{enumerate}
\item the singularities of $X_{\C}$ are ordinary $k$-fold points,
\item the singular points of $X_{\C}$ are real,
\item for any singular point in $X_{\C}$, the fiber $\pi_{\C}'^{-1}(x)$ is totally real.
\end{enumerate}
\end{prop}

\begin{proof}
Assume $X$ is centrally seminormal. Then $X$ is central by Proposition \ref{curve-cent}, and (1) holds by seminormality of $\Pol(X)$ as explained above.

Assume $X_{\C}$ admits as singularities
two complex conjugated points (which are then non-real ordinary
$k$-fold points). We normalize at these two points 
and we get an algebraic curve $Y$ 
such that the map
$Y\rightarrow X$ is birational, finite, and bijective since the real points are not impacted, but it is not an isomorphism (the map $Y_{\C}\rightarrow X_{\C}$ is not
bijective). This is in contradiction with Proposition \ref{lvreel}, so we have proved (2).

Assume now there exists $x\in\Sing(X)$ such that the
$\pi_{\C}'^{-1}(x)$, which contains at least a real point by centrality of $X$, is
not totally real. We normalize at the point $x$
and then we glue together the real points over $x$ (as explained in
\cite{Se}). We get a real algebraic curve $Y$
such that the map
$Y\rightarrow X$ is birational, finite, bijective but not an
isomorphism since $Y_{\C}\rightarrow X_{\C}$ is again not bijective. By Proposition \ref{lvreel} we get a contradiction
and it proves (3). 

Assume now the curve $X$ satisfies the three properties of the
proposition. Remark that (3) implies that $X$ is central. From (2) and (3), it follows that $X$ is
centrally seminormal if and only if $X$ is seminormal. Then (1) implies that $X_{\C}$ is a seminormal algebraic curve by \cite{Bo}. As a consequence $\Pol(X_{\C})$ is
seminormal, which implies that 
$\Pol(X)$ is seminormal by \cite[Cor. 5.7]{GT} as expected.
\end{proof}

In the situation of Proposition \ref{seminormcurve}, we say that the real curve admits only real ordinary $k$-fold singularities.

\begin{rem}
  \begin{enumerate}
   
    \item Property (2) of Proposition \ref{seminormcurve} is
      equivalent with the property of $X$ to be biregularly normal. 
      \item Notice that it follows from (2) and (3) of Proposition \ref{seminormcurve} that a
      centrally seminormal curve has a totally real normalization.
      \end{enumerate}
\end{rem}

We recover from Proposition \ref{seminormcurve} that the nodal curve
$\Z(y^2-x^2(x+1))$ is centrally seminormal, whereas the cuspidal curve
$\Z(y^2-x^3)$ or the union $\Z(xy(x-y))$ of three lines in a plane are not centrally seminormal.

The following examples illustrate the fact that all three conditions in Proposition \ref{seminormcurve} are necessary, even for irreducible curves.

\begin{ex}
\begin{enumerate}
\item Descartes trifolium $X=\Z((x^2+y^2)^2-x(x^2-3y^2))$ is not centrally seminormal by condition (1).

\item The curve $\Z(y^2-(x^2+1)^2x)$ is seminormal but not centrally seminormal, since it has complex singularities therefore does not satisfy condition $(2)$.
\item The plane curve
  $X=\Z((x^2+y^2)^2-x(x^2+3y^2))$ admits a unique singular point at the origin, with three distinct complex tangents only one of which is real. The rational function $f=y^3/x$ is integral over $\Pol(X)$ and admits a continuous extension at the origin. The seminormalization $Y$ of $X$, given by $\Pol(Y)=\Pol(X)[y^3/x]$, is seminormal and central, but not centrally seminormal because the fiber of the map $X'_{\C}\to Y_{\C}$
  over the singular point of $Y_{\C}$ is not totally real.

\end{enumerate}
\end{ex}

The notion of central seminormalization enables to initiate the
classification of central real algebraic sets up to finite birational
homeomorphism (a finite birational polynomial map which is an
homeomorphism for the Euclidean topology), by establishment of the classification in the case of curves. This classification can be seen as a bridge between the topological study of (smooth) real algebraic sets, personalized for curves by the first part of the famous Hilbert sixteen problem, and birational geometry. 

\begin{thm} Let $X\subset \R^n$ be a central algebraic curve. Then $X$
  is finitely birationally homeomorphic to a central algebraic curve with only real ordinary $k$-fold points.
\end{thm}

\begin{proof}
The finite birational model homeomorphic to $X$ having only ordinary $k$-fold real singularities is provided by its central seminormalization.
\end{proof}

\begin{rem} \begin{enumerate}
\item We can provide an almost similar statement with non-necessarily central
  curves, by performing the central seminormalization at the central
  singular points, and performing the seminormalization at the
  isolated singular points. This approach uses the fact that the
  classical and central seminormalizations are local notions for curves.\\
\item A consequence of the classification is that a topologically
  smooth real algebraic curve is finitely birationally homeomorphic to a nonsingular real algebraic curve.
\end{enumerate}
\end{rem}

\vskip 10mm


We end the paper with the promised example showing that there does not exist in general a biggest algebraic set in finite birational bijection with a given algebraic set, contrary to the situation in complex algebraic geometry.

\begin{ex}\label{ex-final} We are going to construct a curve $X$ which does not admit a  biggest algebraic curve in finite birational bijection with itself. We produce $X$ from a nonsingular real curve $Y$ by contracting pairs of complex conjugated points. The key point is to choose for $Y$ a curve without
non-trivial automorphism, like for instance $Y=\PP_{\R}^1\setminus
\{4\;\rm{real\;points}\}$.

The construction goes as follows. Let $(P_1,\overline{P_1})$ and
$(P_2,\overline{P_2})$ be two distinct couples of conjugated points of
$Y_{\C}$. Let $Y_i$ be the nodal curve
obtained from $Y$ by gluing together $P_i$ and $\overline{P_i}$, for $i\in\{1,2\}$, and
$X$ be the curve obtained from $Y$ by gluing together the four points
$P_1,P_2,\overline{P_1},\overline{P_2}$. The curves $Y_1,Y_2$ and $X$ are real curves
with a unique singular point which is an isolated real point, and $Y$ is their common normalization :  
$$\begin{array}{ccccc}
Y&\rightarrow&Y_1\\
\downarrow&&\downarrow\\
Y_2&\rightarrow&X.\\
\end{array}$$
Assume there exists a biggest element $Z$ among the real
algebraic curves with a finite birational map onto $X$ that is bijective in restriction to the real points. Remark that $Z$ has necessarily a unique singular point which is real and isolated.

By assumption on $Z$, there exist finite birational maps $Z\to Y_1$ and $Z\to Y_2$, and they induce bijections at the level of complex points. Moreover $Y_1$ and
$Y_2$ are seminormal (as complex curves) and
thus the maps $Z\to Y_1$ and $Z\to
Y_2$ are automatically isomorphisms. We derive an isomorphism $\varphi:Y_1\to Y_2$
which necessarily maps the unique singular point of $Y_1$ onto the unique singular
point of $Y_2$.

We are going to lift $\varphi$ to a non-trivial automorphism of $Y$, obtaining therefore a contradiction. Denote by $\pi_1:Y\to Y_1$ and $\pi_2:Y\to Y_2$ the
normalization maps. Since $\varphi\circ\pi_1:Y\to Y_2$ is finite and
birational, there exists a finite and birational map $\psi:Y\to Y$ such that
$\pi_2=\varphi\circ\pi_1\circ\psi$ by the universal property of the normalization of  $Y_2$. Note that
$\psi(\{P_2,\overline{P_2}\})=\{P_1,\overline{P_1}\}$, so that $\psi$ is non trivial. Moreover it is an automorphism, whose inverse is constructed similarly by applying the universal property of the normalization of $Y_1$ to the map $\varphi^{-1}\circ\pi_2 : Y \to Y_1$.
\end{ex}


\end{document}